\pdfoutput=1
\documentclass[12pt]{amsart}
\usepackage{float}
\usepackage{amsmath,amsfonts,amssymb,mathabx,latexsym,mathtools}
\usepackage{enumitem}
\usepackage[usenames, dvipsnames]{xcolor}
\usepackage{hyperref}
\usepackage{ytableau}
\usepackage{centernot}
\usepackage{tikz-cd}
\usepackage{tikz}
\usetikzlibrary{decorations.markings}
\usepackage{stmaryrd}
\usepackage{wrapfig}
 \usepackage{url}
\usetikzlibrary{calc, shapes, backgrounds,arrows,positioning,plotmarks}
\tikzset{>=stealth',
  head/.style = {fill = white, text=black},
  plaque/.style = {draw, rectangle, minimum size = 10mm}, 
  pil/.style={->,thick},
  junct/.style = {draw,circle,inner sep=0.5pt,outer sep=0pt, fill=black}
  }
  
\newtheorem{theorem}{Theorem}[section]
\newtheorem{lemma}[theorem]{Lemma}

\newtheorem{proposition}[theorem]{Proposition}
\newtheorem{corollary}[theorem]{Corollary}

\theoremstyle{definition}

\newtheorem{examplex}[theorem]{Example}
\newenvironment{example}
  {\pushQED{\qed}\examplex}
  {\popQED\endexamplex}

\theoremstyle{remark}
\newtheorem{remark}[theorem]{Remark}

\numberwithin{equation}{section}

\usepackage[colorinlistoftodos]{todonotes}

\newcommand \mdCR[1]{{\tt md}(#1)}
\newcommand \maCR[1]{{\tt ma}(#1)}

\newcommand{\code}{{\sf code}}

\newcommand\mydef[1]{{\bf #1}}

\newcommand{\FSVT}{{\sf FSVT}}
\newcommand{\FSVD}{{\sf FSVD}}
\newcommand{\maxexcited}{D_{\tt bot}}

\newcommand{\SVT}{{\sf SVT}}

\newcommand{\Ess}{{\sf Ess}}

\newcommand{\NE}{{\sf NE}}

\newcommand{\SE}{{\sf SE}}

\newcommand{\val}{{\rm val}}

\newcommand{\dem}{\delta}

\newcommand{\Pipes}{{\sf Pipes}}

\newcommand{\KPipes}{\overline{\sf Pipes}}

\newcommand{\sv}{{\sf sv}}

\newcommand{\trunc}{{\sf trunc}}

\newcommand{\excited}{{\sf ExcitedYD}}
\newcommand{\kexcited}{{\sf KExcitedYD}}

\newcommand{\wt}{{\tt wt}}

\newcommand{\GL}{{\rm GL}}

\newcommand{\Sym}{S}

\newcommand{\reg}{{\rm reg}}

\newcommand{\maxsizediag}{\rho_d}
\newcommand{\maxsizeantidiag}{\rho_a}

\newcommand{\diagstat}{f_d}

\newcommand{\antidiagstat}{f_a}

\begin{document}

\title[Regularity of ladder determinantal ideals]{ Castelnuovo-Mumford regularity of\\ ladder determinantal varieties and \\patches of Grassmannian Schubert varieties}

\author{Jenna Rajchgot }
\address[JR]{
Dept.~of Mathematics and Statistics, 
McMaster University, 
Hamilton, ON  \newline \indent  L8S 4K1, CANADA}
\email{rajchgoj@mcmaster.ca}
\thanks{Jenna Rajchgot was partially supported by NSERC Grant RGPIN-2017-05732.}

\author{Colleen Robichaux}
\address[CR]{
Dept.~of Mathematics,
University of California, Los Angeles
Los Angeles, CA \newline \indent 90095, USA}
\email{robichaux@ucla.edu}
\thanks{Colleen Robichaux was supported by the NSF GRFP under Grant No. DGE 1746047 and NSF RTG Grant No. DMS 1937241.}

\author{Anna Weigandt}
\address[AW]{Dept.~of Mathematics, Massachusetts Institute of Technology, Cambridge, MA 02142, USA}
\email{weigandt@mit.edu}
\thanks{Anna Weigandt was partially supported by Bill Fulton's Oscar Zariski Distinguished Professor Chair funds.}

\date{December 1, 2022}

\begin{abstract}
 We give degree formulas for Grothendieck polynomials indexed by vexillary permutations and $1432$-avoiding permutations via tableau combinatorics.  These formulas generalize a formula  for degrees of symmetric Grothendieck polynomials which appeared in previous  joint work of the authors with Y.~Ren and A.~St.~Dizier.
 
 We apply our formulas to compute Castelnuovo-Mumford regularity of classes of generalized determinantal ideals. In particular, we give combinatorial formulas for the regularities of all one-sided mixed ladder determinantal ideals. 
We also derive formulas for the regularities of certain Kazhdan-Lusztig ideals, including those coming from open patches of Schubert varieties in Grassmannians. This provides a correction to a conjecture of Kummini-Lakshmibai-Sastry-Seshadri (2015).  
\end{abstract}

\maketitle
\section{Introduction}
\label{sec:intro}
\emph{Castelnuovo-Mumford regularity} is a fundamental measure of the complexity of a graded module. 
In this paper, we use Schubert calculus techniques to provide explicit, easy-to-compute, combinatorial formulas for the Castelnuovo-Mumford regularity of  classes of generalized determinantal ideals.  The classes we treat include one-sided mixed ladder determinantal ideals and ideals defining patches of Schubert varieties in Grassmannians. These two classes of ideals are connected by work of  N. Gonciulea and C. Miller \cite{GM}.

Let $\Bbbk$ be a field. Let $S = \Bbbk[x_1,\dots, x_n]$ be a polynomial ring with the standard grading, $\text{deg}(x_i) = 1$, and let  $I\subseteq S$ be a homogeneous ideal. 
When $S/I$ is Cohen-Macaulay, as is the case throughout this paper, the regularity of $S/I$ is known to satisfy 
\begin{equation}\label{eq:regFormula}
    \reg(S/I)=\deg  K(S/I;t)-{\rm ht}_S(I),
\end{equation}
 where $K(S/I;t)$ is the \emph{K-polynomial} of $S/I$ and ${\rm ht}_S(I)$ is the height of $I$ in $S$.  Using this fact, the authors, in joint work with Y.~Ren and A.~St.~Dizier \cite{RRRSW}, gave a combinatorial formula which computes the regularity of coordinate rings of \emph{Grassmannian matrix Schubert varieties}.  The key technical ingredient was a  formula of C.~Lenart \cite{Le00} regarding \emph{symmetric Grothendieck polynomials}. In the present paper, we extend our work from \cite{RRRSW}.

\subsection{Summary of results}
  We give a combinatorial formula for degrees of Grothendieck polynomials indexed by \emph{1432-avoiding} permutations (see Theorem~\ref{thm:1432}) and \emph{vexillary (2143-avoiding)} permutations (see Theorem~\ref{thm:vexDeg}).  
These formulas in turn, allow us to compute the regularity of the corresponding matrix Schubert varieties.  In particular, our formula in the vexillary setting provides a formula for Castelnuovo-Mumford regularity of  one-sided mixed ladder determinantal varieties (see Section~\ref{sec:regLad} for details).  Our formulas naturally generalize the Grassmannian formula of \cite{RRRSW} (see Section~\ref{subsec:grConn} for details).

\begin{theorem}
\label{thm:1stmainTheorem}
	Given $w\in \Sym_n$ so that $w$ is 1432-avoiding, \[\reg(S/I_w)=\sum_{k=1}^n \maxsizediag(\sigma_k(w)).\]
\end{theorem}
\noindent Defined in Section~\ref{sec:background}, these $\sigma_k(w)$ are certain subsets of the \emph{Rothe diagram} $D(w)$ of $w$, and $\rho_d(\sigma_k(w))$ denotes the size of the largest diagonal in $\sigma_k(w)$.

\begin{example}\label{ex:triangDecomp1432}
Let $w=1462375$.  In the images below, the elements of $\sigma_k(w)$ are shaded, with a maximal diagonal path in $\sigma_k(w)$ marked with $\times$ for $k\in[5]$. For $k=6,7$, we have $\sigma_k(w)=\emptyset$, and so we omit the figures.  
\[
\begin{tikzpicture}[scale=.3]
	\draw (0,0) rectangle (7,7);
	\draw[fill=gray!50] (1,5) rectangle (2,6);
	\draw [fill=gray!50](1,4) rectangle (2,5);
	\draw [fill=gray!50](2,5) rectangle (3,6);
	\draw[fill=gray!50] (2,4) rectangle (3,5);
	\draw [fill=gray!50](4,4) rectangle (5,5);
	\draw[fill=gray!50] (4,1) rectangle (5,2);
	\filldraw (0.5,6.5) circle (.5ex);
	\draw[line width = .2ex] (0.5,0) -- (0.5,6.5) -- (7,6.5);
	\filldraw (3.5,5.5) circle (.5ex);
	\draw[line width = .2ex] (3.5,0) -- (3.5,5.5) -- (7,5.5);
	\filldraw (5.5,4.5) circle (.5ex);
	\draw[line width = .2ex] (5.5,0) -- (5.5,4.5) -- (7,4.5);
	\filldraw (1.5,3.5) circle (.5ex);
	\draw[line width = .2ex] (1.5,0) -- (1.5,3.5) -- (7,3.5);
	\filldraw (2.5,2.5) circle (.5ex);
	\draw[line width = .2ex] (2.5,0) -- (2.5,2.5) -- (7,2.5);
	\filldraw (6.5,1.5) circle (.5ex);
	\draw[line width = .2ex] (6.5,0) -- (6.5,1.5) -- (7,1.5);
	\filldraw (4.5,0.5) circle (.5ex);
	\draw[line width = .2ex] (4.5,0) -- (4.5,0.5) -- (7,0.5);
	\node[] at (3.5,-1.5){$k=1$};
    \node[] at (1.5,5.5){$\times$};
	\node[] at (2.5,4.5){$\times$};
	\node[] at (4.5,1.5){$\times$};
\end{tikzpicture}
\hspace{1.5em}
\begin{tikzpicture}[scale=.3]
	\draw (0,0) rectangle (7,7);
	\draw (1,5) rectangle (2,6);
	\draw (1,4) rectangle (2,5);
	\draw (2,5) rectangle (3,6);
	\draw (2,4) rectangle (3,5);
	\draw[fill=gray!50] (4,4) rectangle (5,5);
	\draw[fill=gray!50] (4,1) rectangle (5,2);
	\filldraw (0.5,6.5) circle (.5ex);
	\draw[line width = .2ex] (0.5,0) -- (0.5,6.5) -- (7,6.5);
	\filldraw (3.5,5.5) circle (.5ex);
	\draw[line width = .2ex] (3.5,0) -- (3.5,5.5) -- (7,5.5);
	\filldraw (5.5,4.5) circle (.5ex);
	\draw[line width = .2ex] (5.5,0) -- (5.5,4.5) -- (7,4.5);
	\filldraw (1.5,3.5) circle (.5ex);
	\draw[line width = .2ex] (1.5,0) -- (1.5,3.5) -- (7,3.5);
	\filldraw (2.5,2.5) circle (.5ex);
	\draw[line width = .2ex] (2.5,0) -- (2.5,2.5) -- (7,2.5);
	\filldraw (6.5,1.5) circle (.5ex);
	\draw[line width = .2ex] (6.5,0) -- (6.5,1.5) -- (7,1.5);
	\filldraw (4.5,0.5) circle (.5ex);
	\draw[line width = .2ex] (4.5,0) -- (4.5,0.5) -- (7,0.5);
	\node[] at (3.5,-1.5){$k=2$};
	\node[] at (4.5,4.5){$\times$};
\end{tikzpicture}
\hspace{1.5em}
\begin{tikzpicture}[scale=.3]
	\draw (0,0) rectangle (7,7);
	\draw (1,5) rectangle (2,6);
	\draw (1,4) rectangle (2,5);
	\draw (2,5) rectangle (3,6);
	\draw (2,4) rectangle (3,5);
	\draw (4,4) rectangle (5,5);
	\draw (4,1) rectangle (5,2);
	\filldraw (0.5,6.5) circle (.5ex);
	\draw[line width = .2ex] (0.5,0) -- (0.5,6.5) -- (7,6.5);
	\filldraw (3.5,5.5) circle (.5ex);
	\draw[line width = .2ex] (3.5,0) -- (3.5,5.5) -- (7,5.5);
	\filldraw (5.5,4.5) circle (.5ex);
	\draw[line width = .2ex] (5.5,0) -- (5.5,4.5) -- (7,4.5);
	\filldraw (1.5,3.5) circle (.5ex);
	\draw[line width = .2ex] (1.5,0) -- (1.5,3.5) -- (7,3.5);
	\filldraw (2.5,2.5) circle (.5ex);
	\draw[line width = .2ex] (2.5,0) -- (2.5,2.5) -- (7,2.5);
	\filldraw (6.5,1.5) circle (.5ex);
	\draw[line width = .2ex] (6.5,0) -- (6.5,1.5) -- (7,1.5);
	\filldraw (4.5,0.5) circle (.5ex);
	\draw[line width = .2ex] (4.5,0) -- (4.5,0.5) -- (7,0.5);
	\node[] at (3.5,-1.5){$k=3$};
\end{tikzpicture}
\hspace{1.5em}
\begin{tikzpicture}[scale=.3]
	\draw (0,0) rectangle (7,7);
	\draw (1,5) rectangle (2,6);
	\draw (1,4) rectangle (2,5);
	\draw (2,5) rectangle (3,6);
	\draw (2,4) rectangle (3,5);
	\draw[] (4,4) rectangle (5,5);
	\draw[fill=gray!50] (4,1) rectangle (5,2);
	\filldraw (0.5,6.5) circle (.5ex);
	\draw[line width = .2ex] (0.5,0) -- (0.5,6.5) -- (7,6.5);
	\filldraw (3.5,5.5) circle (.5ex);
	\draw[line width = .2ex] (3.5,0) -- (3.5,5.5) -- (7,5.5);
	\filldraw (5.5,4.5) circle (.5ex);
	\draw[line width = .2ex] (5.5,0) -- (5.5,4.5) -- (7,4.5);
	\filldraw (1.5,3.5) circle (.5ex);
	\draw[line width = .2ex] (1.5,0) -- (1.5,3.5) -- (7,3.5);
	\filldraw (2.5,2.5) circle (.5ex);
	\draw[line width = .2ex] (2.5,0) -- (2.5,2.5) -- (7,2.5);
	\filldraw (6.5,1.5) circle (.5ex);
	\draw[line width = .2ex] (6.5,0) -- (6.5,1.5) -- (7,1.5);
	\filldraw (4.5,0.5) circle (.5ex);
	\draw[line width = .2ex] (4.5,0) -- (4.5,0.5) -- (7,0.5);
	\node[] at (3.5,-1.5){$k=4$};
	\node[] at (4.5,1.5){$\times$};
\end{tikzpicture}
\hspace{1.5em}
\begin{tikzpicture}[scale=.3]
	\draw (0,0) rectangle (7,7);
	\draw (1,5) rectangle (2,6);
	\draw (1,4) rectangle (2,5);
	\draw (2,5) rectangle (3,6);
	\draw (2,4) rectangle (3,5);
	\draw[] (4,4) rectangle (5,5);
	\draw[fill=gray!50] (4,1) rectangle (5,2);
	\filldraw (0.5,6.5) circle (.5ex);
	\draw[line width = .2ex] (0.5,0) -- (0.5,6.5) -- (7,6.5);
	\filldraw (3.5,5.5) circle (.5ex);
	\draw[line width = .2ex] (3.5,0) -- (3.5,5.5) -- (7,5.5);
	\filldraw (5.5,4.5) circle (.5ex);
	\draw[line width = .2ex] (5.5,0) -- (5.5,4.5) -- (7,4.5);
	\filldraw (1.5,3.5) circle (.5ex);
	\draw[line width = .2ex] (1.5,0) -- (1.5,3.5) -- (7,3.5);
	\filldraw (2.5,2.5) circle (.5ex);
	\draw[line width = .2ex] (2.5,0) -- (2.5,2.5) -- (7,2.5);
	\filldraw (6.5,1.5) circle (.5ex);
	\draw[line width = .2ex] (6.5,0) -- (6.5,1.5) -- (7,1.5);
	\filldraw (4.5,0.5) circle (.5ex);
	\draw[line width = .2ex] (4.5,0) -- (4.5,0.5) -- (7,0.5);
	\node[] at (3.5,-1.5){$k=5$};
	\node[] at (4.5,1.5){$\times$};
\end{tikzpicture}\]
Theorem~\ref{thm:1stmainTheorem} computes $\reg( S/I_w)=3+1+0+1+1=6$.
\end{example}

Theorem~\ref{thm:1stmainTheorem} is a direct consequence of the following:
\begin{theorem}\label{thm:1432}
	If $w\in \Sym_n$ is $1432$-avoiding, then
	\[\deg(\mathfrak G_w)=\#D(w) + \sum_{k=1}^n \maxsizediag(\sigma_k(w)).\]
\end{theorem} 
For $w\in S_n$, $\#D(w)$ is the Coxeter length of $w$. See Section~\ref{sec:background} for the definitions of Grothendieck polynomials $\mathfrak G_w$ and Rothe diagrams $D(w)$.
The proof of Theorem~\ref{thm:1432} appears in Section~\ref{sec:proofs}.

\begin{example}
Returning to $w$ as in Example~\ref{ex:triangDecomp1432},
Theorem~\ref{thm:1432} with Theorem~\ref{thm:1stmainTheorem} give that 
\begin{align*}
\deg(\mathfrak G_w)&= \#D(w) + \reg( S/I_w)= 6 + (3+1+0+1+1)=12.\qedhere\end{align*}
\end{example}

We have similar diagrammatic regularity and degree formulas in the vexillary setting.

\begin{theorem} \label{thm:2ndmainTheorem}
	Given $v\in \Sym_n$ so that $v$ is vexillary, \[\reg(S/I_v)=\sum_{k=1}^n \maxsizeantidiag(\tau_k(v)).\]
\end{theorem}
\noindent Here, the $\tau_k(v)$ are certain subsets of the {Young diagram} $\lambda(v)$ associated to $v$ and $\rho_a(\tau_k(v))$ denotes the size of the largest antidiagonal in $\tau_k(v)$.  See Section~\ref{sec:background} for details.

\begin{example}
\label{ex:adV}
Let $v=169247358$.
In the diagrams below, the elements of $\tau_i(v)$ are shaded for $i=1,2,3$ in $\lambda(v)$.  In particular, a maximal size antidiagonal path contained in $\tau_i(v)$ has been marked with $\times$'s. 
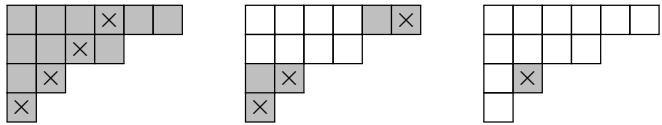
\[
\ytableausetup
{boxsize=0.9em}
\begin{ytableau}
 *(gray!50) &  *(gray!50) &  *(gray!50)& *(gray!50)\times  &*(gray!50) & *(gray!50)   \\
 *(gray!50)&  *(gray!50)&  *(gray!50)\times & *(gray!50)\\
  *(gray!50)&  *(gray!50)\times \\
 *(gray!50)\times 
\end{ytableau}
\hspace{2em}
\begin{ytableau}
\, & \, & \,& \, &  *(gray!50) &*(gray!50)\times   \\
\, & \, & \, & \,\\
*(gray!50) &  *(gray!50)\times \\
 *(gray!50)\times 
\end{ytableau}
\hspace{2em}
\begin{ytableau}
\,  & \, & \, & \, & \, & \,  \\
\,  & \, & \, & \, \\
\, &  *(gray!50)\times \\
\,
\end{ytableau}
\]
Applying Theorem~\ref{thm:2ndmainTheorem}, we have $\reg(S/I_v)=4+3+1=8$.
\end{example}

Theorem~\ref{thm:2ndmainTheorem} is a direct consequence of the following:

\begin{theorem}\label{thm:vexDeg}
	Suppose $v\in \Sym_n$ is vexillary. Then 
	\[\deg(\mathfrak G_v)=\#D(v)+\sum_{i=1}^n\maxsizeantidiag(\tau_k(v)). \]
\end{theorem}

The proof of Theorem~\ref{thm:vexDeg} appears in Section~\ref{sec:proofs}.
\begin{example}
Returning to $v$ as in Example~\ref{ex:adV},
Theorem~\ref{thm:vexDeg} with Theorem~\ref{thm:2ndmainTheorem} give that 
\begin{align*}
\deg(\mathfrak G_v)&= \#D(v) + \reg( S/I_v)= 13 + (4+3+1)=21.\qedhere\end{align*}
\end{example}

We also provide formulas for the regularity of certain homogeneous \emph{Kazhdan-Lusztig ideals} $J_{v,w}$.
When $v$ is a $321$-avoiding permutation, we provide a formula in terms of \emph{pipe dreams} (see Proposition~\ref{prop:pipeKLreg}).  
When $v$ and $w$ are both Grassmannian, we provide an easily computable formula by computing the degree of the corresponding K-polynomial in terms of a vexillary Grothendieck polynomial. 

\subsection{Connections to the literature}
Concurrently with this work the third author, with O.~Pechenik and D.~Speyer \cite{PSW}, derived a combinatorial formula for the regularity of matrix Schubert varieties indexed by arbitrary permutations.  In contrast with our diagrammatic combinatorics, the formula in \cite{PSW} is phrased in terms of a new statistic on permutations. 
In work released around the same time as the present paper,  E.~Hafner \cite{Hafner} obtained a new proof of the vexillary case of \cite{PSW} in terms of bumpless pipe dreams. Her results illustrate the connection from the formula in \cite{PSW} to our vexillary formula through bumpless pipe dreams. A.~Yong also has recent work related to the present paper, where he studies regularities of tangent cones of Schubert varieties \cite{YongReg}.

For certain mixed ladder determinantal ideals, regularity formulas can be deduced through $a$-invariant formulas of S.~Ghorpade and C.~Krattenthaler \cite{GK15}. The ladder determinantal ideals they consider have certain restrictions on their rank conditions. Consequently, their one-sided ideals are special cases of the ideals that we consider. See Section~\ref{sec:regLad} for further discussion.

Our formula for regularity of patches of Grassmannian Schubert varieties (Theorem \ref{thm:KLSSCorrection}) provides a correction to a conjecture of M. Kummini, V. Lakshmibai, P. Sastry, and C. S. Seshadri (see \cite[Conjecture 7.5]{KLSS}). This correction was conjectured in our previous paper (see \cite[Conjecture 5.6]{RRRSW}).

\subsection{Outline of the paper}
In Section~\ref{sec:background}, we introduce the necessary combinatorial background. In Section~\ref{section:tableau}, we give tableau interpretations of our Grothendieck degree formulas. 
We provide proofs of our degree formulas for vexillary and 1432-avoiding permutations in Section~\ref{sec:proofs}.
Section~\ref{section:CM} describes the connection between Grothendieck polynomials and regularity and proves our main theorems. 
Section~\ref{sec:vexGr} applies these regularity formulas to correct the conjecture of \cite{KLSS}. 
Section~\ref{sec:regLad} further applies our main theorems to give combinatorial formulas for the regularity of one-sided ladder determinantal ideals.

\subsection*{Acknowledgements}
We would like to thank  Philippe Nadeau and Alexander Yong for helpful comments and conversations. We would also like to thank Elisa Gorla for helpful communications about the literature. Finally, we would like to thank the anonymous referee for their helpful comments.

\section{Combinatorial degree formulas}
\label{sec:background}

\subsection{Grothendieck polynomials}
\label{subsec:grothbackground}
 
	We start by defining \emph{Grothendieck polynomials}, introduced by A. Lascoux and M.~P. Sch\"utzenberger \cite{LS82} in their study of the K-theory of the complete flag variety.
	Let $\Sym_n$ denote the \mydef{symmetric group} on $n$ letters, i.e., the set of bijections from $[n]:=\{1,2,\ldots, n\}$ to itself.  We write permutations in one-line notation unless otherwise specified and define $w_i:=w(i)$ for $i\in[n]$. 
	The symmetric group $S_n$ acts on $\mathbb Z[x_1,\dots,x_n]$ by $w\cdot f(x_1,\ldots,x_n)=f(x_{w_1},x_{w_2},\dots,x_{w_n})$.
	Let $s_i\in S_n$ be the simple transposition $(i \ i+1)$,  written here in cycle notation.  
	For $f\in\mathbb{Z}[x_1,x_2,\ldots,x_n]$, define
 \[\partial_if=\frac{f-s_if}{x_i-x_{i+1}}, \text{ and } \pi_if=\partial_i(1-x_{i+1})f.\]
  
	We recursively define Grothendieck polynomials as follows.	Let  $w_0=n \, n-1 \, \ldots \, 1$ be the \mydef{longest permutation} in $\Sym_n$. Define 
   \[\mathfrak G_{w_0}(\mathbf x)=\mathfrak G_{w_0}(x_1,x_2,\ldots,x_n)=x_1^{n-1}x_2^{n-2}\cdots x_{n-1}.\]
   For $w\neq w_0$ there exists some $i\in[n-1]$ such that $w_i>w_{i+1}$. Then we define $\mathfrak G_{ws_i}(\mathbf x)=\pi_i(\mathfrak G_w(\mathbf x))$.
	 Since the $\pi_i$ satisfy the same braid and commutation relations as the simple transpositions, $\mathfrak G_w(\mathbf x)$ is well defined.

Write $x\oplus y:=x+y-xy$.  We define the \mydef{double Grothendieck polynomials} using the same recurrence, starting from the initial condition 
\[\mathfrak G_{w_0}(\mathbf x;\mathbf y)=\prod_{1<i+j\leq n}(x_i\oplus y_j).\]  Here, the $\partial_i$'s only act on the $x_i$'s, leaving the $y_j$'s fixed.

\subsection{Permutations}
	First we recall some background on the symmetric group with \cite{Manivel} as a reference.  The \mydef{permutation matrix} of $w$, which we also denote by $w$, is the $0,1$-matrix with $1$'s at $(i,w_i)$ for all $i\in[n]$ and $0$'s elsewhere. 	To each permutation we associate a \mydef{rank function} defined by \[r_w(i,j)=\#\{(k,w_{k}) \,:  \, k\leq i,w_{k} \leq j \}.\]
	
	The \mydef{Rothe diagram} of $w\in \Sym_n$ is the subset
	\[ D(w)= \{(i, j)\in [n]\times [n] \,:  \,  w_i > j \text{ and } w^{-1}_j > i\}.\]
	Visually, $D(w)$ is the set of cells remaining in the $n\times n$ grid after plotting the points $(i,w_i)$ for each $i\in[n]$ and striking out any cells which appear weakly below or weakly to the right of these points, as shown in Example~\ref{ex:Rothe}.

Let $\ell(w):=\# D(w)$ denote the {\bf Coxeter length} of $w$.
The \mydef{code} of $w$ is the tuple $\code(w)=(c_1,\ldots,c_n)$ where $c_i$ records the number of cells in the $i$th row of $D(w)$. Let 
\[{L}(\code(w)):=\max\{i\in[n] \,:  \, c_i>0\}.\]
The \mydef{essential set} of $w$ is the subset of $D(w)$
\[\Ess(w)=\{(i,j)\in D(w) \,:  \, (i+1,j),(i,j+1)\not \in D(w)\}.\] 
The dominant component ${\sf Dom}(w)$ is the connected component of $D(w)$ containing $(1,1)$.

\begin{example}\label{ex:Rothe}
For 
$w=72416835 \in {S}_{8}$, $D(w)$ is the following: 
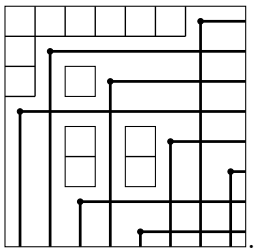
\[
\begin{tikzpicture}[scale=.4]
\draw (0,0) rectangle (8,8);

\draw[line width = .1ex] (0,7)--(6,7);
\draw[line width = .1ex] (1,8) -- (1,5);
\draw[line width = .1ex] (2,8) -- (2,7);
\draw[line width = .1ex] (3,8) -- (3,7);
\draw[line width = .1ex] (4,8) -- (4,7);
\draw[line width = .1ex] (5,8) -- (5,7);
\draw[line width = .1ex] (6,8) -- (6,7);

\draw[line width = .1ex] (0,6) -- (1,6);
\draw[line width = .1ex] (0,5) -- (1,5);

\draw (2,5) rectangle (3,6);

\draw (2,2) rectangle (3,4);
\draw[line width = .1ex] (2,3) -- (3,3);

\draw (4,2) rectangle (5,4);
\draw[line width = .1ex] (4,3) -- (5,3);

\filldraw (6.5,7.5) circle (.5ex);
\draw[line width = .2ex] (6.5,0) -- (6.5,7.5) -- (8,7.5);
\filldraw (1.5,6.5) circle (.5ex);
\draw[line width = .2ex] (1.5,0) -- (1.5,6.5) -- (8,6.5);
\filldraw (3.5,5.5) circle (.5ex);
\draw[line width = .2ex] (3.5,0) -- (3.5,5.5) -- (8,5.5);
\filldraw (0.5,4.5) circle (.5ex);
\draw[line width = .2ex] (0.5,0) -- (0.5,4.5) -- (8,4.5);
\filldraw (5.5,3.5) circle (.5ex);
\draw[line width = .2ex] (5.5,0) -- (5.5,3.5) -- (8,3.5);
\filldraw (7.5,2.5) circle (.5ex);
\draw[line width = .2ex] (7.5,0) -- (7.5,2.5) -- (8,2.5);
\filldraw (2.5,1.5) circle (.5ex);
\draw[line width = .2ex] (2.5,0) -- (2.5,1.5) -- (8,1.5);
\filldraw (4.5,0.5) circle (.5ex);
\draw[line width = .2ex] (4.5,0) -- (4.5,0.5) -- (8,0.5);
\end{tikzpicture}.
\]
Here, we have
$\Ess(w)=\{(1,6),(3,1),(3,3),(6,3),(6,5)\}$, $\code(w)=(6,1,2,0,2,2,0,0)$, and ${\sf Dom}(w)=\{(1,i):i\in [6]\}\cup \{(2,1), (3,1)\}$.
\end{example}

A subset $D\subseteq [n]\times[n]$ is a \mydef{diagonal path} if \[D=\{(i_1,j_1),\ldots,(i_k,j_k)\,:  \, i_1<i_2<\cdots <i_k \text{ and } j_1<j_2<\cdots <j_k\}.\]
 Given $S\subseteq [n]\times[n]$ write $\maxsizediag(S)$ for the size of the largest diagonal path in contained $S$.
 
 A permutation $w\in \Sym_n$ is \mydef{$1432$-avoiding} if there does not exist a \emph{$1432$ pattern}, i.e., indices
$h<i<j<k$ such that $w$ has the pattern $w_h<w_k<w_j<w_i$. 
For example, $w=\underline{2}3\underline{7}1\underline{5}8\underline{4}6$ is not $1432$-avoiding; we underlined the positions of a $1432$ pattern. 
For $w$ $1432$-avoiding, let 
\[\sigma_k(w)=\{ (i,j)\in D(w) \,:  \, i>k \text{ and } j>w_k\},\] i.e., $\sigma_k(w)$ is the set of cells in $D(w)$ which are strictly southeast of $(k,w_k)$.  Example~\ref{ex:triangDecomp1432} gives an example of diagonal paths in $\sigma_k(w)$.



A \mydef{partition} $\lambda=(\lambda_1,\dots,\lambda_k)$ is a weakly decreasing sequence of non negative integers.  We write $|\lambda|=\lambda_1+\dots+\lambda_k$.  The \mydef{Young diagram} of a partition $\lambda$ is the set $\{(i,j)\in \mathbb Z_{>0}\times \mathbb Z_{>0}:1\leq j\leq \lambda_i\}$.  We often conflate Young diagrams with their partitions. Given partitions $\lambda$ and $\mu$, we write $\lambda\subseteq \mu$ to mean that the Young diagram of $\lambda$ is contained in the Young diagram of $\mu$.

Given $w\in S_n$, let $\mu(w)$ be the partition whose Young diagram is \[\bigcup_{(i,j)\in D(w)}[1,i]\times[1,j],\]
i.e., $\mu(w)$ is the smallest partition whose Young diagram contains $D(w)$.

A permutation $v\in \Sym_n$ is \mydef{vexillary} if it does not contain a \emph{$2143$ pattern}, i.e., indices
$i<j<k<l$ such that $v_j<v_i<v_l<v_k$. 
For example, $v=7\underline{2}58\underline{1}3\underline{6}\underline{4}$ is not vexillary since the underlined indices form a $2143$ pattern. 
A vexillary permutation $v$ has \mydef{shape} $\lambda(v)$, where $\lambda(v)$ is
$\code(v)=(c_1,\ldots,c_n)$ sorted into decreasing order.
A vexillary permutation $v$ has \mydef{flag}
\begin{align*}
    \phi(v)&=(\phi_1\leq \phi_2\leq\dots\leq \phi_m), \text{where}\\
    \phi_i=\max\{j \,:  \, &(j,k)\in \mu(v) \ \text{ lies in the same diagonal as } (i,\lambda_i(v))\}.
\end{align*}
Note that we can think of $\lambda(v)$ as the partition with
 the property that each of the diagonals of its Young diagram has the same number of cells as the corresponding diagonal of $D(v)$.
 Observe that $\phi(v)$ tells us how the positions of these boxes changed between $D(v)$ to $\lambda(v)$.

Fill the diagonals of $\lambda(v)$ with $r_v(i,j)$ for the corresponding cells $(i,j)\in D(v)$, so that the entries are (weakly) increasing along diagonals.  Write $F_v$ for this filling (see Figure~\ref{figure:incvex} for an example). Let \[\tau_k(v)=\{(i,j)\in \lambda(v) \,:  \, F_v(i,j)\geq k\}.\]

\ytableausetup{boxsize=1em}
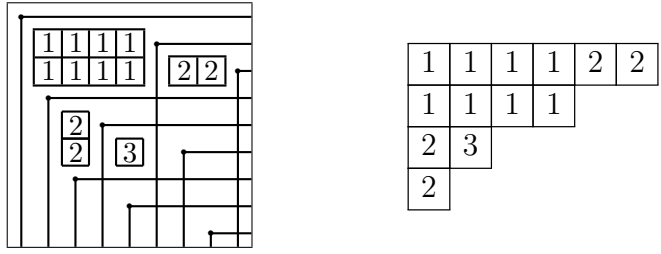
\begin{figure}
\label{fig:vex}
	\[\setlength{\unitlength}{.18mm}
	\begin{picture}(380,180)
		\put(0,15){\framebox(180,180)}
		\thicklines
		\put(10,185){\circle*{4}}
		\put(10,185){\line(1,0){170}}
		\put(10,185){\line(0,-1){170}}
		\put(110,165){\circle*{4}}
		\put(110,165){\line(1,0){70}}
		\put(110,165){\line(0,-1){150}}
		\put(170,145){\circle*{4}}
		\put(170,145){\line(1,0){10}}
		\put(170,145){\line(0,-1){130}}
		\put(30,125){\circle*{4}}
		\put(30,125){\line(1,0){150}}
		\put(30,125){\line(0,-1){110}}
		\put(70,105){\circle*{4}}
		\put(70,105){\line(1,0){110}}
		\put(70,105){\line(0,-1){90}}
		\put(130,85){\circle*{4}}
		\put(130,85){\line(1,0){50}}
		\put(130,85){\line(0,-1){70}}
		\put(50,65){\circle*{4}}
		\put(50,65){\line(1,0){130}}
		\put(50,65){\line(0,-1){50}}
		\put(90,45){\circle*{4}}
		\put(90,45){\line(1,0){90}}
		\put(90,45){\line(0,-1){30}}
		\put(150,25){\circle*{4}}
		\put(150,25){\line(1,0){30}}
		\put(150,25){\line(0,-1){10}}
		
		\put(20,135){\framebox(80,40)}
		\put(40,135){\line(0,1){40}}
		\put(60,135){\line(0,1){40}}
		\put(80,135){\line(0,1){40}}
		\put(20,155){\line(1,0){80}}

		\put(120,135){\framebox(40,20)}
		\put(140,135){\line(0,1){20}}

		\put(40,75){\line(0,1){40}}
		\put(40,95){\line(1,0){20}}
		\put(60,75){\line(0,1){40}}
		\put(40,95){\line(1,0){20}}
		\put(40,115){\line(1,0){20}}
		\put(40,75){\line(1,0){20}}

		\put(80,75){\line(1,0){20}}
		\put(80,95){\line(1,0){20}}
		\put(80,75){\line(0,1){20}}
		\put(100,75){\line(0,1){20}}

		\put(25,159){\small${1}$}
		\put(45,159){\small$1$}
		\put(65,159){\small$1$}
		\put(85,159){\small$1$}
		\put(25,138){\small$1$}
		\put(45,138){\small$1$}
		\put(65,138){\small$1$}
		\put(85,138){\small$1$}
		
		\put(125,138){\small$2$}
		\put(145,138){\small$2$}
		\put(45,98){\small$2$}
		\put(45,78){\small$2$}
		
		\put(85,78){\small$3$}
		
		\put(280,135){
		\ytableausetup
{boxsize=1.3em}
			\begin{ytableau}
1  & 1 & 1 & 1 & 2 & 2  \\
1  & 1 & 1 & 1 \\
2 &  3 \\
2
\end{ytableau}
			}
	\end{picture}
	\]
	\caption{Let $v=169247358$.  Pictured on the left is the filling of the cells $(i,j)\in D(v)$ with $r_v(i,j)$.  On the right is $\lambda(v)$ filled with $F_v$.}
	\label{figure:incvex}
\end{figure}

A subset $A\subseteq [n]\times[n]$ is an \mydef{antidiagonal path} if \[A=\{(i_1,j_1),\ldots,(i_k,j_k)\,:  \, i_1<i_2<\cdots <i_k \text{ and } j_1>j_2>\cdots >j_k\}.\]  Given $S\subseteq [n]\times[n]$ write $\maxsizeantidiag(S)$ for the largest antidiagonal path in $S$. See Example~\ref{ex:adV} for an example of antidiagonal paths in $\tau_k(v)$.



A permutation $g\in S_n$ is \mydef{Grassmannian} if it has a unique descent, i.e. a unique $k\in [n-1]$ such that $g_k>g_{k+1}$.

\subsection{Pipe complexes}
\label{sec:pipeCompl}

Let $a=(a_1,\ldots,a_k)$ be a word on the alphabet $[n-1]$.  We say $a$ is a \mydef{reduced word} for $w$ if $w=s_{a_1}\cdots s_{a_k}$ and $\ell(w)=k.$ 

Define an algebra over $\mathbb Z$ with generators $\{e_w:w\in S_n\}$ and multiplication given by
\[e_we_{s_i}=
\begin{cases}
e_{ws_i} &\text{if } \ell(ws_i)>\ell(w) \\
e_w & \text{if } \ell(ws_i)<\ell(w).
\end{cases}\]
The \mydef{Demazure product} $\dem(a)$ of a word $a=(a_1,\ldots,a_k)$ is defined by computing \[e_{s_{a_1}}\cdots e_{s_{a_k}}=e_{\dem(a)}.\]

Label cells in $D(v)$ along rows so that the first cell in row $i$
 is labeled $i$, the next $i+1$, and so on.
 Given $P\subseteq D(v)$ let $a_P$ be the word obtained by reading the labels of the elements in $P$  within rows from right to left, starting at the top row and working downwards. 
 Let \[\Pipes(v,w)=\{P\subseteq D(v)\,:  \,a_P \text{ is a reduced word for } w\}.\]  Likewise, let \[\KPipes(v,w)=\{P\subseteq D(v)\,:  \,\dem(a_P)=w\}.\]
 For any $P\subseteq [n]\times [n]$, we assign it the $\mathbf t$-weight \[\wt_{\mathbf t}(P)=\prod_{(i,j)\in P}t_{ij}.\]
 Pictorially, we represent $P\subseteq D(v)$ by marking $(i,j)\in D(v)$ with a $+$ whenever $(i,j)\in P$.

 We define the \mydef{unspecialized Grothendieck polynomial} to be 
 \begin{equation}
 \label{eq:unspecgroth}
     \mathfrak G_{v,w}(\mathbf t)=\sum_{P\in\KPipes(v,w)}(-1)^{\#P-\ell(w)}\wt_{\mathbf t}(P).
 \end{equation}

Note that by setting $v=w_0$, we can recover double Grothendieck polynomials by specializing the variables in $\mathfrak G_{w_0,w}(\mathbf t)$:
\[\mathfrak G_w(\mathbf x;\mathbf y)=\mathfrak G        _{w_0,w}(x_1\oplus y_1,x_2\oplus y_2,\dots,x_n\oplus y_n).\]

\subsection{Excited Young Diagrams}\label{sec:EYD}
Fix partitions $\lambda\subseteq \mu$.  Let \[D_{\tt top}(\mu,\lambda)=\{(i,j)\,:  \, i\in[k] \text{ and } j\in [\lambda_i]\}.\] 
We call $D\subseteq \mu$ a \mydef{diagram}, represented graphically by marking these cells in $D$ with $+$'s.
An \mydef{excited move} is a mutation of a local $2\times 2$ subsquare of the form
\begin{equation}\label{eq:EYDsimple}
    \ytableausetup{boxsize=1em}
\begin{ytableau}
+&\\
&
\end{ytableau}
\hspace{1em}
 \raisebox{-.2em}{$\mapsto$}
 \hspace{1em}
\begin{ytableau}
\, &\\
&+\\
\end{ytableau}.
\end{equation}
 Here, the mutated subsquare must be entirely contained within $\mu$.

     We write 
$\excited(\mu,\lambda)$ for the set of $D\subseteq \mu$ which can be obtained by a sequence of excited moves starting from $D_{\tt top}(\mu,\lambda)$.  Such diagrams are called \mydef{excited Young diagrams}.
There is a unique element of $\excited(\mu,\lambda)$ to which no excited moves may be applied, see e.g., \cite[Lemma 7.4]{Weigandt.BPD}.  Call this $\maxexcited(\mu,\lambda)$.  

We also consider \mydef{K-theoretic excited moves} of the form
\begin{equation}\label{eq:EYDkTh}
    \ytableausetup{boxsize=1em}
\begin{ytableau}
+&\\
&
\end{ytableau}
\hspace{1em}
 \raisebox{-.2em}{$\mapsto$}
 \hspace{1em}
\begin{ytableau}
+ &\\
&+\\
\end{ytableau},
\end{equation}
again, where all cells pictured are contained in $\mu$.
Write $\kexcited(\mu,\lambda)$ for the set of diagrams which can be obtained from $D_{\tt top}(\mu,\lambda)$ by a sequence of excited and K-theoretic excited moves.
We weight $D\in\kexcited(\mu,\lambda)$ by \[\wt(D)=\prod_{(i,j)\in D}(x_i\oplus y_j).\]

\begin{proposition}
\label{prop:vexkexcited}
If $v\in S_n$ is vexillary, then \[\mathfrak G_v(\mathbf x;\mathbf y)=\sum_{D\in \kexcited(\mu(v),\lambda(v))}(-1)^{\# D-|\lambda(v)|}\,\wt(D). \]
\end{proposition}
\begin{proof}
This follows by noting the flagged set-valued tableaux (and diagonal pipe dreams) of \cite{KMY} can be identified with $\kexcited(\mu(v),\lambda(v))$.  See e.g., \cite{Graham.Kreiman} and \cite{Weigandt.BPD} for further details.
\end{proof}

\begin{lemma}\label{lem:grassToVex}
Fix partitions $\lambda\subseteq \mu$.  There exists a unique permutation $v\in S_\infty$ so that $D(v)=\maxexcited(\mu,\lambda)$.  In particular, $v$ is vexillary.
\end{lemma}

\begin{proof}
That $\maxexcited(\mu,\lambda)\in\excited(\mu,\lambda)$ is the diagram of a vexillary permutation follows from \cite[Proposition~7.6]{Weigandt.BPD}.
\end{proof}

\begin{example}  Let $\lambda=(5,4,2,1,0)$ and $\mu=(6,6,4,4,4)$.  Then  $v=5713624$ is the unique vexillary permutation so that $D(v)=\maxexcited(\mu,\lambda)$. 
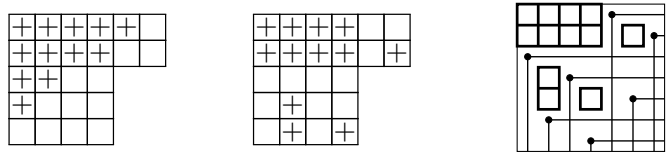
\[\setlength{\unitlength}{.13mm}
\begin{picture}(600,150)
\put(-50,105){\ytableausetup
		{boxsize=0.8em}
		\ytableausetup{notabloids}
		\begin{ytableau}
		+  & +  & +  & +  & +  &  \  \\
		+  & +  & +  & +  & \  &  \  \\
		+  & +  &  \  & \  \\
		+  &  \   &  \  & \   \\
		\  &  \   &  \  & \   
		\end{ytableau} }

	\put(200,105){\ytableausetup
		{boxsize=0.8em}
		\ytableausetup{notabloids}
		\begin{ytableau}
		+  & +  & +  & +  & \  &  \  \\
		+  & +  & +  & +  & \  &  +  \\
		\  & \  &  \  & \  \\
		\  &  +   &  \  & \   \\
		\  &  +   &  \  & +   
		\end{ytableau} }

\put(490,-10){\begin{tikzpicture}[scale=.28]
\draw (0,0) rectangle (7,7);

\draw[line width = .2ex] (0,5) rectangle (4,7);
\draw[line width = .2ex] (6,5) rectangle (5,6);
\draw[line width = .2ex] (1,2) rectangle (2,4);
\draw[line width = .2ex] (4,2) rectangle (3,3);
\draw[line width = .2ex] (3,5) -- (3,7);
\draw[line width = .2ex] (2,5) -- (2,7);
\draw[line width = .2ex] (1,5) -- (1,7);
\draw[line width = .2ex] (0,6) -- (4,6);
\draw[line width = .2ex] (2,3) -- (1,3);

\filldraw (4.5,6.5) circle (.8ex);
\draw[line width = .1ex] (4.5,0) -- (4.5,6.5) -- (7,6.5);
\filldraw (6.5,5.5) circle (.8ex);
\draw[line width = .1ex] (6.5,0) -- (6.5,5.5) -- (7,5.5);
\filldraw (0.5,4.5) circle (.8ex);
\draw[line width = .1ex] (0.5,0) -- (0.5,4.5) -- (7,4.5);
\filldraw (2.5,3.5) circle (.8ex);
\draw[line width = .1ex] (2.5,0) -- (2.5,3.5) -- (7,3.5);
\filldraw (5.5,2.5) circle (.8ex);
\draw[line width = .1ex] (5.5,0) -- (5.5,2.5) -- (7,2.5);
\filldraw (1.5,1.5) circle (.8ex);
\draw[line width = .1ex] (1.5,0) -- (1.5,1.5) -- (7,1.5);
\filldraw (3.5,0.5) circle (.8ex);
\draw[line width = .1ex] (3.5,0) -- (3.5,0.5) -- (7,0.5);
\end{tikzpicture}}

\end{picture}
\]
Above are $D_{\tt top}(\mu,\lambda)$, $\maxexcited(\mu,\lambda)$, and $D(v)$, respectively.
\end{example}

\begin{theorem}\label{thm:grassToVex}
Fix Grassmannian permutations $g$ and $u$ with descent at position $k$ so that $\lambda(g)\subseteq \lambda(u)$.  Let $v$ be the vexillary permutation such that $D(v)=\maxexcited(\lambda(u),\lambda(g))$.  Then
\[\deg(\mathfrak G_{u,g}(\mathbf t))=\deg(\mathfrak G_v(\mathbf x)).\]
\end{theorem}
\begin{proof}
Write $\code(u)=(c_1,\ldots,c_n)$.
Since $u$ is Grassmannian with descent at position $k$, $\lambda(u)=(c_k,c_{k-1},\ldots,c_1)$ (see \cite[Section~2.2]{Manivel}).  
In particular, this means we can identify each cell in $D(u)$ with cells in $D_{\lambda(u)}=\{(i,j)\,:  \,1\leq j\leq c_i\}$ by left justifying cells in $D(u)$ within rows.

Under this identification, we map each element of $\KPipes(u,g)$ to a subset of $D_{\lambda(u)}$.  Call this set of diagrams $L$.  It is immediate that $L\subseteq \KPipes(w_0,g)$.  In particular, this implies elements of $L$ are connected by (flipped) K-theoretic moves, i.e., replacements of the form:
\[
\begin{ytableau}
\,&\\
+&
\end{ytableau}
\hspace{1em}
 \raisebox{-.2em}{$\mapsto$}
 \hspace{1em}
\begin{ytableau}
\, &+\\
&\\
\end{ytableau}
\]
and
\[
\begin{ytableau}
\,&\\
+&
\end{ytableau}
\hspace{1em}
 \raisebox{-.2em}{$\mapsto$}
 \hspace{1em}
\begin{ytableau}
\, &+\\
+&\\
\end{ytableau}.
\]
By flipping the first $k$ rows vertically, we see that elements of $L$ are in bijection with elements of $\kexcited(\lambda(u),\lambda(g))$.  Thus, we have a (degree preserving) bijection between elements of $\KPipes(u,g)$  and $\kexcited(\lambda(u),\lambda(g))$.

Then by Equation~\eqref{eq:unspecgroth} and Proposition~\ref{prop:vexkexcited}, we conclude $\deg(\mathfrak G_{u,g}(\mathbf t))=\deg(\mathfrak G_v(\mathbf x))$.
\end{proof}

\begin{example}
Let $g=1247356$ and $u=1457236$.  An element of $\KPipes(u,g)$ and its corresponding K-theoretic excited Young diagram are pictured below.
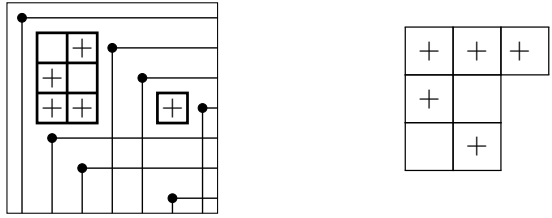
\[\begin{tikzpicture}[scale=.4]
\draw (0,0) rectangle (7,7);
\draw[line width = .2ex] (1,3) rectangle (3,6);
\draw[line width = .2ex] (6,3) rectangle (5,4);
\draw[line width = .2ex] (2,3) -- (2,6);
\draw[line width = .2ex] (1,4) -- (3,4);
\draw[line width = .2ex] (1,5) -- (3,5);

\filldraw (0.5,6.5) circle (.8ex);
\draw[line width = .1ex] (0.5,0) -- (0.5,6.5) -- (7,6.5);
\filldraw (3.5,5.5) circle (.8ex);
\draw[line width = .1ex] (3.5,0) -- (3.5,5.5) -- (7,5.5);
\filldraw (4.5,4.5) circle (.8ex);
\draw[line width = .1ex] (4.5,0) -- (4.5,4.5) -- (7,4.5);
\filldraw (6.5,3.5) circle (.8ex);
\draw[line width = .1ex] (6.5,0) -- (6.5,3.5) -- (7,3.5);
\filldraw (1.5,2.5) circle (.8ex);
\draw[line width = .1ex] (1.5,0) -- (1.5,2.5) -- (7,2.5);
\filldraw (2.5,1.5) circle (.8ex);
\draw[line width = .1ex] (2.5,0) -- (2.5,1.5) -- (7,1.5);
\filldraw (5.5,0.5) circle (.8ex);
\draw[line width = .1ex] (5.5,0) -- (5.5,0.5) -- (7,0.5);
\node at (1.5,3.5) {$+$};
\node at (2.5,3.5) {$+$};
\node at (5.5,3.5) {$+$};
\node at (1.5,4.5) {$+$};
\node at (2.5,5.5) {$+$};
\end{tikzpicture}
\hspace{5em}
\raisebox{4.5em}{
\ytableausetup
		{boxsize=1.5em}
		\ytableausetup{notabloids}
		\begin{ytableau}
		+& +  & +  \  \\
		 +&\,    \  \\
		\, & +   \\
		\end{ytableau} }
\]
\end{example}
\subsection{Connections to the Grassmannian degree formula}\label{subsec:grConn}
In previous work with Ren and St.~Dizier \cite{RRRSW}, the authors presented a formula to compute the degree of symmetric Grothendieck polynomials.  
If $u\in S_n$ is Grassmannian with descent $k$, then the symmetric Grothendieck polynomial is $\mathfrak G_{u(\lambda)}(x_1,\dots,x_k):=\mathfrak G_u(x_1,\dots,x_n)$.  
Since Grassmannian permutations are both 1432-avoiding and vexillary, our new degree formulas are two different generalizations of this formula.  We illustrate these connections here.

 Write $\delta^{(k)}=(k,k-1,\ldots,1)$.  Let $\sv(\lambda)=\max\{k\,:  \,\delta^{(k)}\subseteq \lambda\}$.  
Given a partition $\lambda=(\lambda_1,\ldots,\lambda_k)$, let $\trunc^{(i)}(\lambda)$ be the partition obtained by removing the first $i$ columns of the Young diagram of $\lambda$.  Then:
\begin{theorem}[{\cite{RRRSW}}]
\label{thm:grassdeg}
If $\lambda=(\lambda_1,\ldots,\lambda_k)$, then
\[\deg(\mathfrak G_\lambda(x_1,\ldots,x_k))=|\lambda|+\sum_{i=1}^k\sv(\trunc^{(\lambda_i)}(\lambda)).\]
\end{theorem}
Theorem~\ref{thm:grassdeg} can be recovered using Theorem~\ref{thm:1432} or Theorem~\ref{thm:vexDeg}. We illustrate this in the example below.

\begin{example}
Let $\lambda=(3,2,2,0)$ and $k=4$. The Grassmannian permutation associated to the pair $(\lambda,k)$ is $w=1457236$.  The first line below computes the formula in Theorem~\ref{thm:grassdeg} where the $i$th Young diagram has $\trunc^{(\lambda_i)}(\lambda)$ shaded, with $\delta^{(k)}$ marked with $\times$'s for $k=\sv(\trunc^{(\lambda_i)}(\lambda))$.  
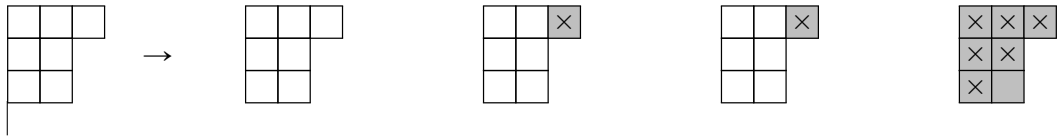
\[\begin{picture}(400,50)
\put(10,30){\ytableausetup
{boxsize=1em}
{\begin{ytableau}
 \ &  \ &  \ \\
  \ &  \ \\
 \ &  \
\end{ytableau}}}
\put(14,-7){\line(0,1){13}}
\put(65,20){$\rightarrow$}
\put(100,30){\ytableausetup
{boxsize=1em}
{\begin{ytableau}
 \ &  \ &  \ \\
  \ &  \ \\
 \ &  \
\end{ytableau}}}
\put(190,30){\ytableausetup
{boxsize=1em}
{\begin{ytableau}
 \ &  \ &  *(gray!50)\times \\
  \ &  \ \\
 \ &  \
\end{ytableau}}}
\put(280,30){\ytableausetup
{boxsize=1em}
{\begin{ytableau}
 \ &  \ &  *(gray!50) \times\\
  \ &  \ \\
 \ &  \
\end{ytableau}}}
\put(370,30){\ytableausetup
{boxsize=1em}
{\begin{ytableau}
  *(gray!50)\times&   *(gray!50)\times&  *(gray!50)\times\\
  *(gray!50)\times&  *(gray!50)\times\\
 *(gray!50)\times&  *(gray!50) 
\end{ytableau}}}
\end{picture}
\]
Below, we demonstrate the rule given in Theorem~\ref{thm:1432}.
 Here, we have $\sigma_k(w)$ shaded, with the longest diagonal marked with $\times$'s.
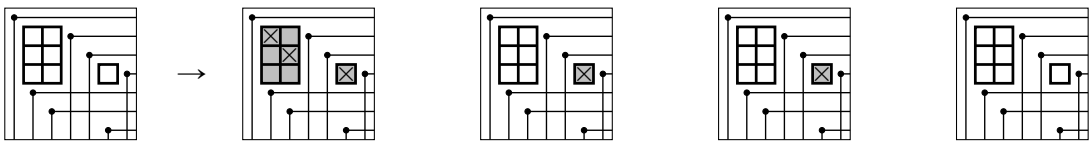
\[\begin{picture}(400,70)
\put(0,10){\begin{tikzpicture}[scale=.25]
\draw (0,0) rectangle (7,7);

\draw[line width = .2ex] (1,3) rectangle (3,6);
\draw[line width = .2ex] (6,3) rectangle (5,4);
\draw[line width = .2ex] (2,3) -- (2,6);
\draw[line width = .2ex] (1,4) -- (3,4);
\draw[line width = .2ex] (1,5) -- (3,5);

\filldraw (0.5,6.5) circle (.8ex);
\draw[line width = .1ex] (0.5,0) -- (0.5,6.5) -- (7,6.5);
\filldraw (3.5,5.5) circle (.8ex);
\draw[line width = .1ex] (3.5,0) -- (3.5,5.5) -- (7,5.5);
\filldraw (4.5,4.5) circle (.8ex);
\draw[line width = .1ex] (4.5,0) -- (4.5,4.5) -- (7,4.5);
\filldraw (6.5,3.5) circle (.8ex);
\draw[line width = .1ex] (6.5,0) -- (6.5,3.5) -- (7,3.5);
\filldraw (1.5,2.5) circle (.8ex);
\draw[line width = .1ex] (1.5,0) -- (1.5,2.5) -- (7,2.5);
\filldraw (2.5,1.5) circle (.8ex);
\draw[line width = .1ex] (2.5,0) -- (2.5,1.5) -- (7,1.5);
\filldraw (5.5,0.5) circle (.8ex);
\draw[line width = .1ex] (5.5,0) -- (5.5,0.5) -- (7,0.5);
\end{tikzpicture}}
\put(65,32){$\rightarrow$}
\put(90,10){\begin{tikzpicture}[scale=.25]
\draw (0,0) rectangle (7,7);

\draw[fill=gray!50,line width = .2ex] (1,3) rectangle (3,6);
\draw[fill=gray!50,line width = .2ex] (6,3) rectangle (5,4);
\draw[line width = .2ex] (2,3) -- (2,6);
\draw[line width = .2ex] (1,4) -- (3,4);
\draw[line width = .2ex] (1,5) -- (3,5);

\filldraw (0.5,6.5) circle (.8ex);
\draw[line width = .1ex] (0.5,0) -- (0.5,6.5) -- (7,6.5);
\filldraw (3.5,5.5) circle (.8ex);
\draw[line width = .1ex] (3.5,0) -- (3.5,5.5) -- (7,5.5);
\filldraw (4.5,4.5) circle (.8ex);
\draw[line width = .1ex] (4.5,0) -- (4.5,4.5) -- (7,4.5);
\filldraw (6.5,3.5) circle (.8ex);
\draw[line width = .1ex] (6.5,0) -- (6.5,3.5) -- (7,3.5);
\filldraw (1.5,2.5) circle (.8ex);
\draw[line width = .1ex] (1.5,0) -- (1.5,2.5) -- (7,2.5);
\filldraw (2.5,1.5) circle (.8ex);
\draw[line width = .1ex] (2.5,0) -- (2.5,1.5) -- (7,1.5);
\filldraw (5.5,0.5) circle (.8ex);
\draw[line width = .1ex] (5.5,0) -- (5.5,0.5) -- (7,0.5);
\node at (1.5,5.5) {$\times$};
\node at (2.5,4.5) {$\times$};
\node at (5.5,3.5) {$\times$};
\end{tikzpicture}}
\put(180,10){\begin{tikzpicture}[scale=.25]
\draw (0,0) rectangle (7,7);

\draw[line width = .2ex] (1,3) rectangle (3,6);
\draw[fill=gray!50,line width = .2ex] (6,3) rectangle (5,4);
\draw[line width = .2ex] (2,3) -- (2,6);
\draw[line width = .2ex] (1,4) -- (3,4);
\draw[line width = .2ex] (1,5) -- (3,5);

\filldraw (0.5,6.5) circle (.8ex);
\draw[line width = .1ex] (0.5,0) -- (0.5,6.5) -- (7,6.5);
\filldraw (3.5,5.5) circle (.8ex);
\draw[line width = .1ex] (3.5,0) -- (3.5,5.5) -- (7,5.5);
\filldraw (4.5,4.5) circle (.8ex);
\draw[line width = .1ex] (4.5,0) -- (4.5,4.5) -- (7,4.5);
\filldraw (6.5,3.5) circle (.8ex);
\draw[line width = .1ex] (6.5,0) -- (6.5,3.5) -- (7,3.5);
\filldraw (1.5,2.5) circle (.8ex);
\draw[line width = .1ex] (1.5,0) -- (1.5,2.5) -- (7,2.5);
\filldraw (2.5,1.5) circle (.8ex);
\draw[line width = .1ex] (2.5,0) -- (2.5,1.5) -- (7,1.5);
\filldraw (5.5,0.5) circle (.8ex);
\draw[line width = .1ex] (5.5,0) -- (5.5,0.5) -- (7,0.5);
\node at (5.5,3.5) {$\times$};
\end{tikzpicture}}
\put(270,10){\begin{tikzpicture}[scale=.25]
\draw (0,0) rectangle (7,7);

\draw[line width = .2ex] (1,3) rectangle (3,6);
\draw[fill=gray!50,line width = .2ex] (6,3) rectangle (5,4);
\draw[line width = .2ex] (2,3) -- (2,6);
\draw[line width = .2ex] (1,4) -- (3,4);
\draw[line width = .2ex] (1,5) -- (3,5);

\filldraw (0.5,6.5) circle (.8ex);
\draw[line width = .1ex] (0.5,0) -- (0.5,6.5) -- (7,6.5);
\filldraw (3.5,5.5) circle (.8ex);
\draw[line width = .1ex] (3.5,0) -- (3.5,5.5) -- (7,5.5);
\filldraw (4.5,4.5) circle (.8ex);
\draw[line width = .1ex] (4.5,0) -- (4.5,4.5) -- (7,4.5);
\filldraw (6.5,3.5) circle (.8ex);
\draw[line width = .1ex] (6.5,0) -- (6.5,3.5) -- (7,3.5);
\filldraw (1.5,2.5) circle (.8ex);
\draw[line width = .1ex] (1.5,0) -- (1.5,2.5) -- (7,2.5);
\filldraw (2.5,1.5) circle (.8ex);
\draw[line width = .1ex] (2.5,0) -- (2.5,1.5) -- (7,1.5);
\filldraw (5.5,0.5) circle (.8ex);
\draw[line width = .1ex] (5.5,0) -- (5.5,0.5) -- (7,0.5);
\node at (5.5,3.5) {$\times$};
\end{tikzpicture}}
\put(360,10){\begin{tikzpicture}[scale=.25]
\draw (0,0) rectangle (7,7);

\draw[line width = .2ex] (1,3) rectangle (3,6);
\draw[line width = .2ex] (6,3) rectangle (5,4);
\draw[line width = .2ex] (2,3) -- (2,6);
\draw[line width = .2ex] (1,4) -- (3,4);
\draw[line width = .2ex] (1,5) -- (3,5);

\filldraw (0.5,6.5) circle (.8ex);
\draw[line width = .1ex] (0.5,0) -- (0.5,6.5) -- (7,6.5);
\filldraw (3.5,5.5) circle (.8ex);
\draw[line width = .1ex] (3.5,0) -- (3.5,5.5) -- (7,5.5);
\filldraw (4.5,4.5) circle (.8ex);
\draw[line width = .1ex] (4.5,0) -- (4.5,4.5) -- (7,4.5);
\filldraw (6.5,3.5) circle (.8ex);
\draw[line width = .1ex] (6.5,0) -- (6.5,3.5) -- (7,3.5);
\filldraw (1.5,2.5) circle (.8ex);
\draw[line width = .1ex] (1.5,0) -- (1.5,2.5) -- (7,2.5);
\filldraw (2.5,1.5) circle (.8ex);
\draw[line width = .1ex] (2.5,0) -- (2.5,1.5) -- (7,1.5);
\filldraw (5.5,0.5) circle (.8ex);
\draw[line width = .1ex] (5.5,0) -- (5.5,0.5) -- (7,0.5);
\end{tikzpicture}}
\end{picture}
\]
Now, we use the formula from Theorem~\ref{thm:vexDeg}.
In each Young diagram, we have shaded $\tau_k(w)$, with the longest antidiagonals marked with $\times$'s.
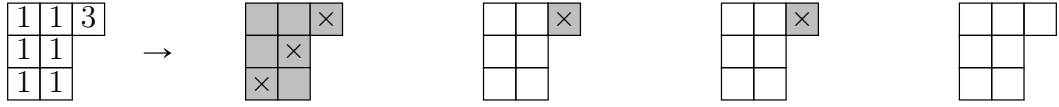
\[\begin{picture}(400,50)
\put(10,30){\ytableausetup
{boxsize=1em}
{\begin{ytableau}
 1 &  1 &  3 \\
 1 &  1 \\
 1 &  1
\end{ytableau}}}
\put(65,20){$\rightarrow$}
\put(100,30){\ytableausetup
{boxsize=1em}
{\begin{ytableau}
  *(gray!50)  &   *(gray!50)  &  *(gray!50)\times\\
  *(gray!50) &  *(gray!50)\times\\
 *(gray!50)\times&  *(gray!50) 
\end{ytableau}}}
\put(190,30){\ytableausetup
{boxsize=1em}
{\begin{ytableau}
 \ &  \ &  *(gray!50)\times \\
  \ &  \ \\
 \ &  \
\end{ytableau}}}
\put(280,30){\ytableausetup
{boxsize=1em}
{\begin{ytableau}
 \ &  \ &  *(gray!50) \times\\
  \ &  \ \\
 \ &  \
\end{ytableau}}}
\put(370,30){\ytableausetup
{boxsize=1em}
{\begin{ytableau}
 \ &  \ &  \ \\
  \ &  \ \\
 \ &  \
\end{ytableau}}}
\end{picture}
\]
 Thus we see all three formulas compute $\deg(\mathfrak G_\lambda(x_1,\ldots,x_k))=|\lambda|+3+1+1=12$.
\end{example}

\section{Tableau formulas for Grothendieck polynomials}
\label{section:tableau}

Since their introduction, Grothendieck polynomials have been studied with a number of combinatorial formulas (\cite{FominKrillov, Le00, Buch}).  For our degree formulas, we will focus on two tableau formulas in the special cases of 1432-avoiding permutations and vexillary permutations.  Furthermore, in each of these cases, we construct a tableau whose weight contributes to the top degree terms of the corresponding Grothendieck polynomial.

\subsection{Set-valued Rothe tableaux}

\label{sec:1432}

A \mydef{set-valued Rothe tableau $T$ of shape} $D(w)$ is a filling 
of $D(w)$ with nonempty subsets of $\mathbb{Z}_{>0}$ such that for boxes $a,b\in D(w)$:
\begin{itemize}
    \item if $a$ lies north of $b$ in the same column, then $\max T(a)<\min T(b)$, and
    \item if $a$ lies west of $b$ in the same row, then $\min T(a)\geq\max T(b)$,
\end{itemize}
where $T(a)$ denotes the set of entries of $T$ in box $a$.
Let $\SVT(D(w))$ be the collection of such tableaux. 
We say a tableau $T \in \SVT(D(w))$ \mydef{is flagged by}
$\phi=(\phi_1,\phi_2,\ldots,\phi_n)$ if for each box $b$ in row $i$ of $D(w)$, $\max{T(b)}\leq \phi_i$ for all $i$. For a $1432$-avoiding $w\in \Sym_n$, let 
\[\FSVD(w)=\{T\in \SVT(D(w)) \,:  \, T \mbox{ is flagged by } (1,2, \ldots,n)\}.\]

\begin{example}\label{ex:FSVT}
Below is some $T\in \FSVD(w)$ for $w=1462375$. 
\[
\begin{tikzpicture}[scale=.5]
\draw (0,0) rectangle (7,7);

\draw (1,6) rectangle (3,4);
\draw[line width = .1ex] (1,5) -- (3,5);
\draw[line width = .1ex] (2,4) -- (2,6);

\draw (4,4) rectangle (5,5);
\draw (4,1) rectangle (5,2);

\filldraw (0.5,6.5) circle (.5ex);
\draw[line width = .2ex] (0.5,0) -- (0.5,6.5) -- (7,6.5);
\filldraw (3.5,5.5) circle (.5ex);
\draw[line width = .2ex] (3.5,0) -- (3.5,5.5) -- (7,5.5);
\filldraw (5.5,4.5) circle (.5ex);
\draw[line width = .2ex] (5.5,0) -- (5.5,4.5) -- (7,4.5);
\filldraw (1.5,3.5) circle (.5ex);
\draw[line width = .2ex] (1.5,0) -- (1.5,3.5) -- (7,3.5);
\filldraw (2.5,2.5) circle (.5ex);
\draw[line width = .2ex] (2.5,0) -- (2.5,2.5) -- (7,2.5);
\filldraw (6.5,1.5) circle (.5ex);
\draw[line width = .2ex] (6.5,0) -- (6.5,1.5) -- (7,1.5);
\filldraw (4.5,0.5) circle (.5ex);
\draw[line width = .2ex] (4.5,0) -- (4.5,0.5) -- (7,0.5);
\put(16,75){$2 1$}
\put(33,75){$1$}

\put(18,60){$3$}
\put(30,60){$3 2$}
\put(59,60){$2 1$}

\put(60,22){$\scriptstyle{6 5}$}
\put(60,16){$\scriptstyle{4 3}$}
\end{tikzpicture}
\]
\end{example}

\begin{theorem}\cite[Theorem 1.1]{Fan.Guo}\label{thm:321svtGroth}
For $w\in \Sym_n$ $1432$-avoiding, $\mathfrak G_{w}$ has the following expansion:
\begin{equation}\label{equation:g321}
\mathfrak G_w(\mathbf{x},\mathbf{y})=\sum_{T\in \FSVD(w)}(-1)^{\#T-\#D(w)}\prod_{e\in T} x_{\val(e)}\oplus y_{\lambda_{r(e)}+\phi_{r(e)}-c(e)-\val(e)+1}.
\end{equation}
where the product is over entries $e$ in $T$ whose value is $\val(e)$ and $c(e),r(e)$ are the column and row indices of $e$.
\end{theorem}

For $T\in \FSVD(w)$, let 
$\#T$ denote the number of entries in $T$. 
 We say $T\in \FSVD(w)$ is \mydef{maximal} if $T'\in \FSVD(w)$ implies $\#T'\leq\#T$.
Now we give a construction of $T_w\in \FSVD(w)$ for a given $1432$-avoiding $w$. Theorem~\ref{thm:1432} proves $T_w$ is maximal. Let $\mdCR{D}$ denote the northmost then westmost maximal diagonal path of $D\subseteq [n]^2$.
For $\mdCR{\sigma_k(w)}\neq \emptyset$, let 
 \begin{align*}
 \NE(\mdCR{\sigma_k(w)})&=\{(i,j)\in D(w)- \mdCR{\sigma_k(w)} \,:  \, (i,j) \text{ lies northeast of } \mdCR{\sigma_k(w)} \}.
\end{align*}
Take $T_0\in \SVT(D(w))$ such that $T_0(i,j)=i$ for $i\in [L(\code(w))]$. 
For $k\in [L(\code(w))-1]$, 
let $T_k\in \SVT(D(w))$ such that for $(i,j)\in D(w)$:
  \[T_k(i,j):=\begin{cases}
    T_{k-1}(i,j)\cup \{\min{T_{k-1}(i,j)}-1\} &\text{if } (i,j)\in \mdCR{\sigma_k(w)},\\
	T_{k-1}(i,j)-1 &\mbox{if }  (i,j)\in \NE(\mdCR{\sigma_k(w)}),\\
		T_{k-1}(i,j) &\text{otherwise,}
	\end{cases}\]
where $T(i,j)-1$ is entrywise subtraction. Let $T_w:=T_{L(\code(w))}$.

\begin{example}\label{ex:buildSkewT}
Below we construct $T_w$ for $w=1462375$. 
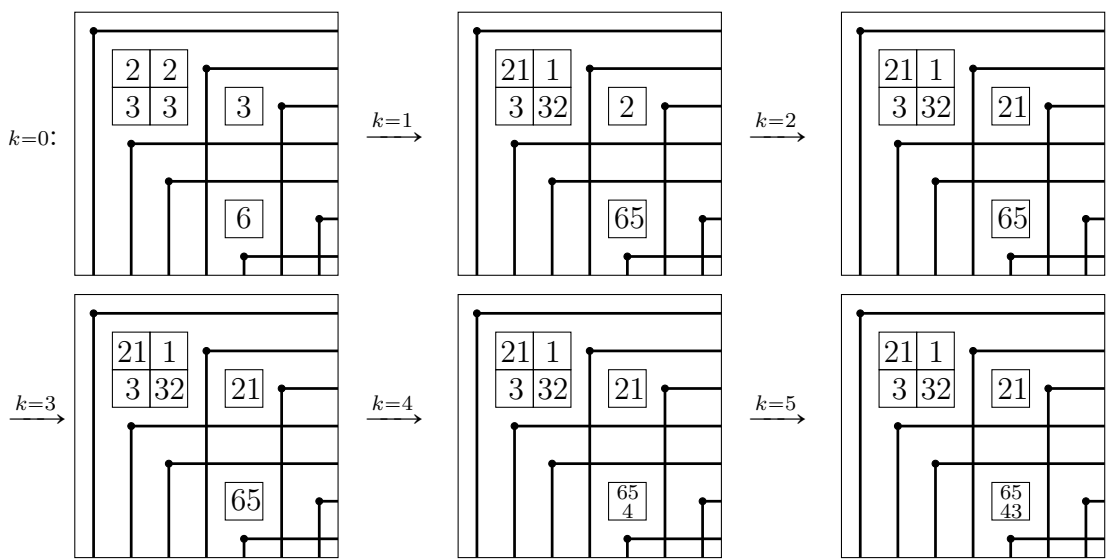
\[\begin{picture}(450,100)
\put(0,50){$\scriptstyle{k=0}$:}
\put(25,0){\begin{tikzpicture}[scale=.5]
\draw (0,0) rectangle (7,7);

\draw (1,5) rectangle (2,6);
\draw (1,4) rectangle (2,5);
\draw (2,5) rectangle (3,6);
\draw (2,4) rectangle (3,5);

\draw (4,4) rectangle (5,5);

\draw (4,1) rectangle (5,2);

\put(19,74){$2$}
\put(33,74){$2$}
\put(19,60){$3$}
\put(33,60){$3$}
\put(61,60){$3$}
\put(61,18){$6$}

\filldraw (0.5,6.5) circle (.5ex);
\draw[line width = .2ex] (0.5,0) -- (0.5,6.5) -- (7,6.5);
\filldraw (3.5,5.5) circle (.5ex);
\draw[line width = .2ex] (3.5,0) -- (3.5,5.5) -- (7,5.5);
\filldraw (5.5,4.5) circle (.5ex);
\draw[line width = .2ex] (5.5,0) -- (5.5,4.5) -- (7,4.5);
\filldraw (1.5,3.5) circle (.5ex);
\draw[line width = .2ex] (1.5,0) -- (1.5,3.5) -- (7,3.5);
\filldraw (2.5,2.5) circle (.5ex);
\draw[line width = .2ex] (2.5,0) -- (2.5,2.5) -- (7,2.5);
\filldraw (6.5,1.5) circle (.5ex);
\draw[line width = .2ex] (6.5,0) -- (6.5,1.5) -- (7,1.5);
\filldraw (4.5,0.5) circle (.5ex);
\draw[line width = .2ex] (4.5,0) -- (4.5,0.5) -- (7,0.5);
\end{tikzpicture}}
\put(135,50){$\xrightarrow{k=1}$}
\put(170,0){\begin{tikzpicture}[scale=.5]
\draw (0,0) rectangle (7,7);

\draw (1,5) rectangle (2,6);
\draw (1,4) rectangle (2,5);
\draw (2,5) rectangle (3,6);
\draw (2,4) rectangle (3,5);

\draw (4,4) rectangle (5,5);

\draw (4,1) rectangle (5,2);

\put(16,74){$2 1$}
\put(33,74){$1$}
\put(19,60){$3$}
\put(30,60){$3 2$}
\put(61,60){$2$}
\put(59,18){$6 5$}

\filldraw (0.5,6.5) circle (.5ex);
\draw[line width = .2ex] (0.5,0) -- (0.5,6.5) -- (7,6.5);
\filldraw (3.5,5.5) circle (.5ex);
\draw[line width = .2ex] (3.5,0) -- (3.5,5.5) -- (7,5.5);
\filldraw (5.5,4.5) circle (.5ex);
\draw[line width = .2ex] (5.5,0) -- (5.5,4.5) -- (7,4.5);
\filldraw (1.5,3.5) circle (.5ex);
\draw[line width = .2ex] (1.5,0) -- (1.5,3.5) -- (7,3.5);
\filldraw (2.5,2.5) circle (.5ex);
\draw[line width = .2ex] (2.5,0) -- (2.5,2.5) -- (7,2.5);
\filldraw (6.5,1.5) circle (.5ex);
\draw[line width = .2ex] (6.5,0) -- (6.5,1.5) -- (7,1.5);
\filldraw (4.5,0.5) circle (.5ex);
\draw[line width = .2ex] (4.5,0) -- (4.5,0.5) -- (7,0.5);
\end{tikzpicture}}
\put(280,50){$\xrightarrow{k=2}$}
\put(315,0){\begin{tikzpicture}[scale=.5]
\draw (0,0) rectangle (7,7);

\draw (1,5) rectangle (2,6);
\draw (1,4) rectangle (2,5);
\draw (2,5) rectangle (3,6);
\draw (2,4) rectangle (3,5);

\draw (4,4) rectangle (5,5);

\draw (4,1) rectangle (5,2);

\put(16,74){$2 1$}
\put(33,74){$1$}
\put(19,60){$3$}
\put(30,60){$3 2$}
\put(59,60){$2 1$}
\put(59,18){$6 5$}

\filldraw (0.5,6.5) circle (.5ex);
\draw[line width = .2ex] (0.5,0) -- (0.5,6.5) -- (7,6.5);
\filldraw (3.5,5.5) circle (.5ex);
\draw[line width = .2ex] (3.5,0) -- (3.5,5.5) -- (7,5.5);
\filldraw (5.5,4.5) circle (.5ex);
\draw[line width = .2ex] (5.5,0) -- (5.5,4.5) -- (7,4.5);
\filldraw (1.5,3.5) circle (.5ex);
\draw[line width = .2ex] (1.5,0) -- (1.5,3.5) -- (7,3.5);
\filldraw (2.5,2.5) circle (.5ex);
\draw[line width = .2ex] (2.5,0) -- (2.5,2.5) -- (7,2.5);
\filldraw (6.5,1.5) circle (.5ex);
\draw[line width = .2ex] (6.5,0) -- (6.5,1.5) -- (7,1.5);
\filldraw (4.5,0.5) circle (.5ex);
\draw[line width = .2ex] (4.5,0) -- (4.5,0.5) -- (7,0.5);
\end{tikzpicture}}
\end{picture}
\]
\[\begin{picture}(450,100)
\put(0,50){$\xrightarrow{k=3}$}
\put(25,0){\begin{tikzpicture}[scale=.5]
\draw (0,0) rectangle (7,7);

\draw (1,5) rectangle (2,6);
\draw (1,4) rectangle (2,5);
\draw (2,5) rectangle (3,6);
\draw (2,4) rectangle (3,5);

\draw (4,4) rectangle (5,5);

\draw (4,1) rectangle (5,2);

\put(16,74){$2 1$}
\put(33,74){$1$}
\put(19,60){$3$}
\put(30,60){$3 2$}
\put(59,60){$2 1$}
\put(59,18){$6 5$}

\filldraw (0.5,6.5) circle (.5ex);
\draw[line width = .2ex] (0.5,0) -- (0.5,6.5) -- (7,6.5);
\filldraw (3.5,5.5) circle (.5ex);
\draw[line width = .2ex] (3.5,0) -- (3.5,5.5) -- (7,5.5);
\filldraw (5.5,4.5) circle (.5ex);
\draw[line width = .2ex] (5.5,0) -- (5.5,4.5) -- (7,4.5);
\filldraw (1.5,3.5) circle (.5ex);
\draw[line width = .2ex] (1.5,0) -- (1.5,3.5) -- (7,3.5);
\filldraw (2.5,2.5) circle (.5ex);
\draw[line width = .2ex] (2.5,0) -- (2.5,2.5) -- (7,2.5);
\filldraw (6.5,1.5) circle (.5ex);
\draw[line width = .2ex] (6.5,0) -- (6.5,1.5) -- (7,1.5);
\filldraw (4.5,0.5) circle (.5ex);
\draw[line width = .2ex] (4.5,0) -- (4.5,0.5) -- (7,0.5);
\end{tikzpicture}}
\put(135,50){$\xrightarrow{k=4}$}
\put(170,0){\begin{tikzpicture}[scale=.5]
\draw (0,0) rectangle (7,7);

\draw (1,5) rectangle (2,6);
\draw (1,4) rectangle (2,5);
\draw (2,5) rectangle (3,6);
\draw (2,4) rectangle (3,5);

\draw (4,4) rectangle (5,5);

\draw (4,1) rectangle (5,2);

\put(16,74){$2 1$}
\put(33,74){$1$}
\put(19,60){$3$}
\put(30,60){$3 2$}
\put(59,60){$2 1$}
\put(60,22){$\scriptstyle{6 5}$}
\put(62,15){$\scriptstyle{4}$}

\filldraw (0.5,6.5) circle (.5ex);
\draw[line width = .2ex] (0.5,0) -- (0.5,6.5) -- (7,6.5);
\filldraw (3.5,5.5) circle (.5ex);
\draw[line width = .2ex] (3.5,0) -- (3.5,5.5) -- (7,5.5);
\filldraw (5.5,4.5) circle (.5ex);
\draw[line width = .2ex] (5.5,0) -- (5.5,4.5) -- (7,4.5);
\filldraw (1.5,3.5) circle (.5ex);
\draw[line width = .2ex] (1.5,0) -- (1.5,3.5) -- (7,3.5);
\filldraw (2.5,2.5) circle (.5ex);
\draw[line width = .2ex] (2.5,0) -- (2.5,2.5) -- (7,2.5);
\filldraw (6.5,1.5) circle (.5ex);
\draw[line width = .2ex] (6.5,0) -- (6.5,1.5) -- (7,1.5);
\filldraw (4.5,0.5) circle (.5ex);
\draw[line width = .2ex] (4.5,0) -- (4.5,0.5) -- (7,0.5);
\end{tikzpicture}}
\put(280,50){$\xrightarrow{k=5}$}
\put(315,0){\begin{tikzpicture}[scale=.5]
\draw (0,0) rectangle (7,7);

\draw (1,5) rectangle (2,6);
\draw (1,4) rectangle (2,5);
\draw (2,5) rectangle (3,6);
\draw (2,4) rectangle (3,5);

\draw (4,4) rectangle (5,5);

\draw (4,1) rectangle (5,2);

\put(16,74){$2 1$}
\put(33,74){$1$}
\put(19,60){$3$}
\put(30,60){$3 2$}
\put(59,60){$2 1$}
\put(60,22){$\scriptstyle{6 5}$}
\put(60,15){$\scriptstyle{4 3}$}

\filldraw (0.5,6.5) circle (.5ex);
\draw[line width = .2ex] (0.5,0) -- (0.5,6.5) -- (7,6.5);
\filldraw (3.5,5.5) circle (.5ex);
\draw[line width = .2ex] (3.5,0) -- (3.5,5.5) -- (7,5.5);
\filldraw (5.5,4.5) circle (.5ex);
\draw[line width = .2ex] (5.5,0) -- (5.5,4.5) -- (7,4.5);
\filldraw (1.5,3.5) circle (.5ex);
\draw[line width = .2ex] (1.5,0) -- (1.5,3.5) -- (7,3.5);
\filldraw (2.5,2.5) circle (.5ex);
\draw[line width = .2ex] (2.5,0) -- (2.5,2.5) -- (7,2.5);
\filldraw (6.5,1.5) circle (.5ex);
\draw[line width = .2ex] (6.5,0) -- (6.5,1.5) -- (7,1.5);
\filldraw (4.5,0.5) circle (.5ex);
\draw[line width = .2ex] (4.5,0) -- (4.5,0.5) -- (7,0.5);
\end{tikzpicture}}
\end{picture}
\]
\end{example}

\begin{lemma}\label{lemma:1432FlagFill}
Suppose $w$ in $\Sym_n$ is $1432$-avoiding. Then $T_w\in \FSVD(w)$.
 \end{lemma}
 
\begin{proof}
We proceed by showing $T_{k}\in \FSVD(w)$ for $k\in[L(\code(w))]$ by induction on $k$.
By construction, $T_0\in \FSVD(w)$.
Suppose $T_{k-1}\in \FSVD(w)$ for some $k\in[L(\code(w))]$. If $\mdCR{\sigma_k(w)}=\emptyset$, the result follows the inductive assumption since $T_k=T_{k-1}$.

Otherwise, since $T_{k-1}\in \FSVD(w)$ by construction of $T_k$, 
\[\max T_k(i,j)\leq \max T_{k-1}(i,j)\leq i.\] 
Similarly since  
$T_{k-1}$ is decreasing along rows, $T_k$ is clearly decreasing along rows. 
By definition of $T_k$, any $(i,j)$ can be decremented no more than $i-1$ times, so no entry can be decremented to $0$.
Thus it remains to show $T_k$ increases down columns.
Consider some $(i,j)\in \mdCR{\sigma_k(w)}$. Let \[i'=\max\{x<i \,:  \, (x,j)\in D(w)\}.\]
Since $T_{k-1}$ is increasing down columns, it suffices to show that $\max T_{k}(i',j)<\min T_{k}(i,j)$. If $(i',j)\in \NE(\mdCR{\sigma_k(w)})$ or does not exist, the result follows by the construction of $T_k$.

 Otherwise, by the definition of $\mdCR{\sigma_k(w)}$, it follows that $i=i'+h$ for some $h>1$. Then for $s\in[i']$, by the definitions of $T_{s}$ and $\mdCR{\sigma_s(w)}$,
\begin{equation}\label{eq:specialTabInc}
\max T_s(i',j)+h\leq\min T_s(i,j).
\end{equation}
Thus if $k\leq i'$, we are done. 
If $k>i'$, it follows that $\max T_{k}(i',j)=\max T_{i'}(i',j)$ and $\min T_{k}(i,j)\geq \min T_{i'}(i,j)-{h+1}$,
so by Equation~(\ref{eq:specialTabInc}), 
\[\max T_k(i',j)\leq\min T_k(i,j)-1.\]
 Thus $T_{k}\in \FSVD(w)$.
\end{proof}

\subsection{Set-valued Young tableaux}

 \begin{figure}
\setlength{\unitlength}{1.5em}
\[
\begin{tikzpicture}[scale=.5]

\draw[line width = .05ex] (0,5)--(8,5);
\draw[line width = .05ex] (0,4)--(8,4);
\draw[line width = .05ex] (0,3)--(5,3);
\draw[line width = .05ex] (0,2)--(5,2);
\draw[line width = .05ex] (0,1)--(5,1);

\draw[line width = .05ex] (1,6) -- (1,0);
\draw[line width = .05ex] (2,6) -- (2,0);
\draw[line width = .05ex] (3,6) -- (3,0);
\draw[line width = .05ex] (4,6) -- (4,0);
\draw[line width = .05ex] (5,6) -- (5,3);
\draw[line width = .05ex] (6,6) -- (6,3);
\draw[line width = .05ex] (7,6) -- (7,3);

\draw[line width = .3ex] (0,6)--(6,6)--(6,5)--(4,5)--(4,4)--(2,4)--(2,3)--(1,3)--(1,2)--(0,2)--(0,6);

\draw[line width = .5ex, gray] (8,6)--(8,3)--(5,3)--(5,0)--(0,0)--(0,6)--(8.09,6);

\filldraw[red] (5.5,5.5) circle (.5ex);
\draw[line width = .1ex,red] (5.5,5.5) -- (8,3);

\filldraw[red] (3.5,4.5) circle (.5ex);
\draw[line width = .1ex,red] (3.5,4.5) -- (5,3);

\filldraw[red] (1.5,3.5) circle (.5ex);
\draw[line width = .1ex,red] (1.5,3.5) -- (5,0);

\filldraw[red] (0.5,2.5) circle (.5ex);
\draw[line width = .1ex,red] (0.5,2.5) -- (3,0);
\end{tikzpicture}
\hspace{3em} 
\ytableausetup
{boxsize=1.5em}
\raisebox{6em}{\begin{ytableau}
 1  & 1 & 1  & 1 & \text{\footnotesize 123} & 3 \\
 2  & 2 & 23  & 3 \\
 34  & 45 \\
56
\end{ytableau}}
\]
\caption{ Let $v$ be as in Figure~\ref{fig:vex}.  Then $\lambda(v)=(6,4,2,1)$ and $\mu(v)=(8,8,8,5,5,5)$. Pictured on the left is $\lambda(v)\subset \mu(v)$ with the diagonals used to compute $\phi(v)=(3,3,6,6)$ drawn in red. To the right is an element of $\FSVT(v)$.}
\label{fig:flag}  
\end{figure}
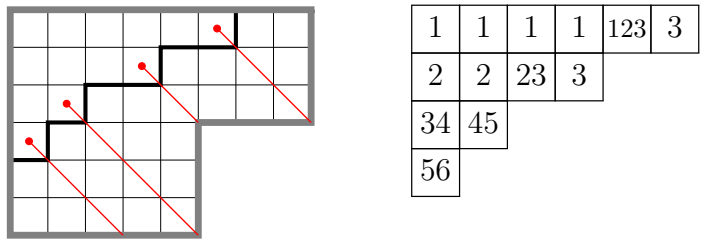

A \mydef{set-valued tableau $T$ of shape} $\lambda=(\lambda_1,\lambda_2,\ldots,\lambda_n)$ is a filling 
of $\lambda$ with nonempty subsets of $\mathbb{Z}_{>0}$ such that for boxes $(a,b)\in\lambda$:
\begin{itemize}
    \item if $a$ lies north of $b$, then $\max T(a)<\min T(b)$, and
    \item if $a$ lies west of $b$, then $\max T(a)\leq\min T(b)$,
\end{itemize}
where $T(a)$  denotes the set of entries of $T$ in box $a$.
Let $\SVT(\lambda)$ be the collection of such tableaux. 
We say a tableau $T \in \SVT(\lambda)$ \mydef{is flagged by}
$\phi=(\phi_1,\phi_2,\ldots,\phi_n)$ if for each box $b$ in row $i$ of $\lambda$, we have $\max{T(b)}\leq \phi_i$. For a vexillary permutation $v$, let 
\[\FSVT(v)=\{T\in \SVT(\lambda(v)) \,:  \, T \text{ is flagged by } \phi(v)\}.\]
An example of some $T\in \FSVT(169247358)$ 
is given in Figure~\ref{fig:flag}.
We note that many different choices of flagging can result in the same underlying set of tableaux.  See \cite[Remark~3.10]{MPP} for further commentary.

\begin{theorem}\cite[Theorem 5.8]{KMY}
If $v\in \Sym_n$ is vexillary, the double Grothendieck polynomial $\mathfrak G_{v}(\mathbf x;\mathbf y)$ has the following expansion:
\begin{equation}\label{equation:gVex}
\mathfrak G_v(\mathbf{x};\mathbf{y})=\sum_{T\in \FSVT(v)}(-1)^{\#T-|\lambda|}\prod_{e\in T}x_{\val(e)}\oplus y_{\val(e)+c(e)-r(e)},
\end{equation}
where the product is over entries in $T$ whose value is $\val(e)$ and $c(e),r(e)$ are the column and row indices of $e$.
\end{theorem}

 For $T\in \FSVT(v)$ let 
$\#T$ denote the number of entries in $T$. We say $T\in \FSVT(v)$ is \mydef{maximal} if $\#T=\max\{\#U:U\in \FSVT(v)\}$.
Now we give a construction of $U_v\in \SVT(v)$ for a given vexillary $v$. Theorem~\ref{thm:vexDeg} proves $U_v$ is maximal.

Let $\maCR{\lambda}$ denote the northmost then westmost maximal antidiagonal path of $\bigcup_i\mu_i\subseteq \lambda$.
For $\maCR{\tau_k(v)}\neq \emptyset$, let 
 \begin{align*}
 \SE(\maCR{\tau_k(v)})&=\{(i,j)\in \lambda- \maCR{\tau_k(v)} \,:  \, (i,j) \text{ lies southeast of } \maCR{\tau_k(v)} \}.
\end{align*}

Take $U_0\in \SVT(\lambda)$ such that $U_0(i,j)=i$ for $i\in[\ell(\lambda)]$. 
For $k\in[\ell(\lambda)-1]$, 
let $U_k\in \SVT(\lambda)$ such that for $(i,j)\in \lambda:$
  \[U_k(i,j):=\begin{cases}
    U_{k-1}(i,j)\cup \{\max{U_{k-1}(i,j)}+1\} &\text{if } (i,j)\in \maCR{\tau_k(v)},\\
	U_{k-1}(i,j)+1 &\mbox{if }  (i,j)\in \SE(\maCR{\tau_k(v)}),\\
		U_{k-1}(i,j) &\text{otherwise,}
	\end{cases}\]
where $U(i,j)+1$ is entrywise addition. Let $U_v:=U_{\ell(\lambda)}$.

\begin{remark}\label{rmk:VexFlagFill}
By a similar argument to Lemma~\ref{lemma:1432FlagFill}, it follows that $U_v\in {\FSVT}(v)$.
By Theorem~\ref{thm:vexDeg}, it follows that $U_v$ is maximal. 
\end{remark}

\begin{example}\label{ex:buildVexT}
Let $v=169247358$.  From Figure~\ref{fig:vex}, we saw $\lambda(v)=(6,4,2,1)$.  Furthermore, $\phi(v)=(3,3,6,6)$. Below is the construction of $U_v$ from $U_0$.
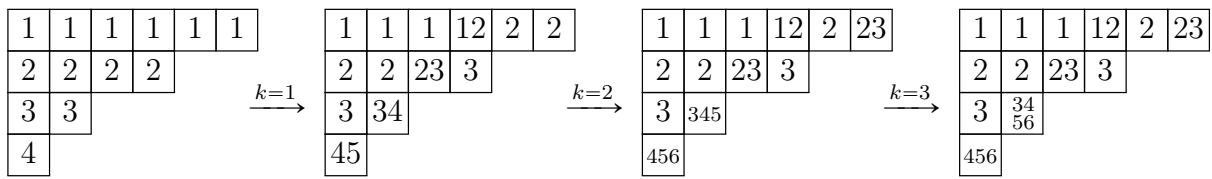
\[\begin{picture}(460,70)
\put(0,-50){\ytableausetup
{boxsize=1.3em}
\raisebox{8.5em}{\begin{ytableau}
 1  & 1 & 1  & 1 & 1 & 1 \\
 2  & 2 & 2  & 2 \\
 3  & 3 \\
 4  
\end{ytableau}}}
\put(95,25){$\xrightarrow{k=1}$}
\put(120,-50){\ytableausetup
{boxsize=1.3em}
\raisebox{8.5em}{\begin{ytableau}
1  & 1 & 1  & 12 & 2 & 2 \\
 2  & 2 & 23  & 3 \\
 3  & 34 \\
 45  
\end{ytableau}}}
\put(215,25){$\xrightarrow{k=2}$}
\put(240,-50){\ytableausetup
{boxsize=1.3em}
\raisebox{8.5em}{\begin{ytableau}
1  & 1 & 1  & 12 & 2 & 23 \\
 2  & 2 & 23  & 3 \\
 3  & \scriptstyle{345} \\
 \scriptstyle{456}  
\end{ytableau}}}
\put(335,25){$\xrightarrow{k=3}$}
\put(360,-50){\ytableausetup
{boxsize=1.3em}
\raisebox{8.5em}{\begin{ytableau}
 1  & 1 & 1  & 12 & 2 & 23 \\
 2  & 2 & 23  & 3 \\
 3  & \ \\
 \scriptstyle{456}
\end{ytableau}}}
\put(384,27){$\scriptstyle{34}$}
\put(384,20){$\scriptstyle{56}$}
\end{picture}
\]
\end{example}

 \section{Proofs of degree formulas}
\label{sec:proofs}

In this section, we prove our Grothendieck degree formulas for 1432-avoiding permutations and vexillary permutations to deduce our main theorems.

\subsection{Proof of Theorem~\ref{thm:1432}}

Recall, $\mdCR{D}$ is the northmost then westmost diagonal path of $D\subseteq [n]^2$
and
 \[\sigma_k(w)=\{ (i,j)\in D(w) \,:  \, i>k, j>w(k)\}.\]
For brevity, define $\diagstat(w)=\#D(w) + \sum_{k=1}^n \maxsizediag(\sigma_k(w))$.

We start by recalling a lemma from \cite{Fan.Guo}.
\begin{lemma}[{\cite[Lemma~2.4]{Fan.Guo}}]
Let $w\neq w_0$ be a 1432-avoiding permutation.  If $r$ is the first ascent of $w$, then $ws_r$ is also 1432-avoiding.
\end{lemma}

\begin{proposition}
	\label{prop:maxtabsize}
	If $w$ is 1432-avoiding, there exists $T\in \FSVD(w)$ such that \[\#T=\diagstat(w).\]  In particular, $\deg(\mathfrak G_w)\geq \diagstat(w)$.
\end{proposition}
\begin{proof}
    This follows by Lemma~\ref{lemma:1432FlagFill} since $\#T_w=\diagstat(w)$ by construction.
\end{proof}

\begin{lemma}
	\label{lemma:maxtabflag}
	Let $w\neq w_0$ be a 1432-avoiding permutation and suppose $r$ is the first ascent of $w$.  If there is a maximal diagonal path in $\sigma_r(w)$ which has no cells in row $r+1$, then there exists a maximal set-valued Rothe tableau for $w$ such that the entries in row $r+1$ restricted to $\sigma_r(w)$ are all strictly less than $r+1$.
\end{lemma}

\begin{proof}
Suppose $w$ is such that there is a maximal diagonal path in $\sigma_r(w)$ which has no cells in row $r+1$. Consider maximal $T\in \FSVD(w)$ such that $T$ has boxes containing $r+1$ in row $r+1$ restricted to $\sigma_r(w)$. 
We will construct $T'\in\FSVD(w)$ such that $\#T'=\#T$ and such that the entries in row $r+1$ restricted to $\sigma_r(w)$ are all strictly less than $r+1$.

Let ${\sf b}_1$ denote the box containing the eastmost occurrence of $r+1$ in row $r+1$ in $T$. For $1<i\leq \#\mdCR{\sigma_r(w)}$, we
define ${\sf b}_i\in \sigma_r(w)$ as the box containing the northmost, then eastmost occurrence of $r+i$ in $T$, in the region strictly east of ${\sf b}_{i-1}$. Thus $\{{\sf b}_i\}_{i\in[\sigma_r(w)]}$ forms a diagonal path.

Let ${\sf c}_i$ denote the northmost box of $\sigma_r(w)$ lying directly south of ${\sf b}_i$ for each $i\in[\#\mdCR{\sigma_r(w)}]$. By the assumption that there is a maximal diagonal path in $\sigma_r(w)$ which has no cells in row $r+1$ and the definition of ${\sf b}_i$, $\{{\sf c}_i\}_{i\in [k]}$ exists for some $1\leq k\leq \#\mdCR{\sigma_r(w)}$.
Let $P$ be constructed as follows:
\[P = \{{\sf b}_1\}\cup \{{\sf b}_i \,:  \, {\sf b}_{i-1}\in P \text{ and } {\sf c}_{i-1} \text{ lies in the same row as }  {\sf b}_{i}\}.\]
Let $P'=\{{\sf c}_i \,:  \, {\sf b}_i\in P\}$. 
By maximality of $T$, it follows that $\{r+i,r+i-1\}\subseteq T({\sf b}_{i})$ for each $i\in[\#P]$.
Take $T'$ such that 
 \[T'(x,y):=\begin{cases}
    T({\sf b}_{i})\setminus\{r+i\} &\text{if } (x,y)={\sf b}_{i},\\
     T({\sf c}_{i})\cup\{r+i\} &\text{if } (x,y)={\sf c}_{i},\\
    T(x,y)-1 &\text{if } (x,y)\text{ lies directly between } {\sf c}_{i} \text{ and } {\sf b}_{i+1},\\
    T(x,y)-1 &\text{if } (x,y)=(r+1,y)\in \sigma_r(w) \text{, lying west of } {\sf b}_{1},\\
		T(x,y) &\text{otherwise.}
	\end{cases}\]
It is straightforward to check $T'\in \FSVD(w)$. 
Since $\#T=\#T'$, $T'$ is of the desired form. 
\end{proof}

 \begin{proposition}\label{prop:diag1432Move}
 Suppose $w\in \Sym_n$ is $1432$-avoiding.
 Let $r$ denote the position of the first ascent of $w$ and $\{c_m<\dots<c_0\}=\{w_r\leq i \leq w_{r+1} \,:  \, (r+1,i)\in D(w)\}$. Then 
    \[D(w\cdot s_r)=\left(D(w)- \{ (r+1,c_i):0\leq i\leq m\}\right )\cup \{ (r,c_i):0\leq i\leq m\}\cup \{(r,w_r)\}.\]
 \end{proposition} 

 \begin{proof}
 This follows by the definition of $D(w)$, since $r$ is the first ascent of $w$.
\end{proof}

\begin{lemma}
	\label{lemma:maxdiag}
	Let $w\neq w_0$ be a 1432-avoiding permutation, and suppose $r$ is the first ascent of $w$.  If there is a maximal diagonal path in $\sigma_r(w)$ which has no cells in row $r+1$, then
	\[\diagstat(w)+1=\diagstat(ws_r).\]  Otherwise, \[\diagstat(w)=\diagstat(ws_r).\]
\end{lemma}
\begin{proof}
    By Proposition~\ref{prop:diag1432Move}, $\#D(ws_r)=\#D(w)+1$. Further, since $r$ was the first ascent of $w$, $(r,w_r)\in{\sf Dom}(ws_r)$. 
    Further we see 
    \[\maxsizediag(\sigma_k(w))=\maxsizediag(\sigma_k(ws_r)) \text{ for } k\neq r,r+1\] by Proposition~\ref{prop:diag1432Move}. By definition of $r$, \[\maxsizediag(\sigma_{r+1}(w))=\maxsizediag(\sigma_{r}(ws_r)).\] Finally, by Proposition~\ref{prop:diag1432Move}, $\sigma_{r+1}(ws_r)=\sigma_{r}(w)- \bigcup_{i=0}^{\ell} (r,c_i)$. Thus $\maxsizediag(\sigma_{r}(w))=\maxsizediag(\sigma_{r+1}(ws_r))$ if 
    there is a maximal diagonal path in $\sigma_r(w)$ which has no cells in row $r+1$. Otherwise, $\maxsizediag(\sigma_{r}(w))=\maxsizediag(\sigma_{r+1}(ws_r))+1$, so the result follows by the definition of $\diagstat$.
\end{proof}

\begin{proof}[Proof of Theorem~\ref{thm:1432}]
	
	We proceed by induction on $\ell(w_0)-\ell(w)$.  In the base case, $w=w_0$ and the formula is immediate since $\deg(w_0)=\ell(w_0)=\#D(w_0)=\diagstat(w_0)$.
	
	Now pick $w\in S_n$ so that $w\neq w_0$.  Assume the formula holds for all $w'\in S_n$ so that $\ell(w')>\ell(w)$.  Let $r$ be the first ascent of $w$.  Let $R$ denote the set of boxes in $\sigma_r(w)$ lying in row $r+1$.
		By Proposition~\ref{prop:diag1432Move}, one obtains $D(ws_r)$ from $D(w)$ by shifting all cells in $R$ up one row and then placing a new cell in position $(r,w_r)$. 
		
		Consider $T\in \FSVD(w)$.
		We will construct $T'\in\FSVD(ws_r)$ from $T$ by the following:
		 \[T'(x-1,y):=\begin{cases}
    r &\text{if } (x-1,y)=(r,w_r),\\
    T(x,y)-\{r+1\}\cup\{r\} &\text{if } (x,y)\in R, \ r+1\in T(x,y), \text{ and } r\not\in T(x,y),\\
    T(x,y)-\{r+1\}&\text{if } (x,y)\in R, \ r+1\in T(x,y), \text{ and } r\in T(x,y), \\
    T(x,y) &\text{if } (x,y)\in R, \ r+1\not\in T(x,y), \\
    T(x-1,y) &\text{otherwise.}
	\end{cases}\]

	Thus $T'\in \FSVD(ws_r)$ and $\#T'\geq\#T$, giving $\deg(\mathfrak G_w)\leq\deg(\mathfrak G_{ws_r})$.
	We have two cases to check.

	\noindent{\sf Case 1:} Suppose all maximal diagonal paths in $\sigma_r(w)$ have a cell in row $r+1$. 
	
We have 
\begin{align*}
	\diagstat(w)&\leq \deg(\mathfrak G_w) &\text{(by Proposition~\ref{prop:maxtabsize})}\\
		&\leq \deg(\mathfrak G_{ws_r}) & \\
		&=\diagstat(ws_r) &\text{(by inductive hypothesis).}
\end{align*}
	By Lemma~\ref{lemma:maxdiag}, $\diagstat(w)=\diagstat(ws_r)$.  Thus, $\diagstat(w)= \deg(\mathfrak G_w)$.

	\noindent{\sf Case 2:} Suppose there exists a maximal diagonal path in $\sigma_r(w)$ which has no cells in row $r+1$.
	By Lemma~\ref{lemma:maxtabflag}, there exists a maximal tableau $T$ for $w$ so that boxes in $R$ have entries less than $r+1$.  Using the above construction for $T'\in \FSVD(ws_r)$, it follows that $\#T'=\#T+1$.  As a consequence, $\deg(\mathfrak G_w)<\deg(\mathfrak G_{ws_r})$.

	Thus,
	\begin{align*}
		\diagstat(w)&\leq \deg(\mathfrak G_w) &\text{(by Proposition~\ref{prop:maxtabsize})}\\
		&<\deg(\mathfrak G_{ws_r}) & \\
		&=\diagstat(ws_r) &\text{(by inductive hypothesis)}\\
			&=\diagstat(w)+1 &\text{(by Lemma~\ref{lemma:maxdiag}).}
	\end{align*}
Thus $\diagstat(w)= \deg(\mathfrak G_w)$.
\end{proof}

\subsection{Proof of Theorem~\ref{thm:vexDeg}}

 If $v$ is vexillary, we associate to $v$ the following statistic:
\begin{equation}
\label{eqn:dv}
\antidiagstat(v)= \#D(v)+\sum_{i=1}^n \maxsizeantidiag(\tau_i(v)).
\end{equation}
Note that by definition, $\#D(v)=|\lambda(v)|.$
The goal of this section is to prove Theorem~\ref{thm:vexDeg}, i.e., to show
if $v$ is vexillary, then $\deg(\mathfrak G_v)=\antidiagstat(v)$.
We start with some lemmas.

 We follow \cite{KMY} for combinatorial background. The \mydef{maximal corner} $(r,s)$ of $w$ is the position of the right most cell in the last row of $D(w)$. Let $t_{i,j}$ denote the transposition $(i \, j)$.
Define $w_P:=wt_{r,w^{-1}(s)}$. Then $w_P$ is the unique permutation such that
\begin{equation}\label{eq:wPdef}
    D(w_P)=D(w)-\{(r,s)\}.
\end{equation}
Grothendieck polynomials satisfy a recurrence known as \mydef{transition}.  
Recall $t_{i,j}$ denotes the transposition $(i,j)$.  Let $w_P:=wt_{r,w^{-1}(s)}$. Let $i_1<i_2<\cdots<i_k$ be the list of those indices $i<r$ for which $\ell(w_P)+1=\ell(w_Pt_{i,r})$.

 \begin{theorem}[{\cite{Lascoux.Transition}}]
\label{thm:transition} Given $w\in S_n$, with maximal corner $(r,s)$ and $t_{i_j,r}$'s as above,
\[ \mathfrak G_w=\mathfrak G_{w_P}+(x_r-1)(\mathfrak G_{w_P}\star(1-t_{i_1,r})(1-t_{i_2,r})\cdots(1-t_{i_k,r})),\] where
$\mathfrak G_v\star u:=\mathfrak G_{vu}$.
\end{theorem}

When $v$ is vexillary, there is at most one index $i<r$ for which $\ell(v_P)+1=\ell(v_Pt_{i,r})$.  When such an index exists, we define $v_C=v_Pt_{i,r}$.  In this case, Theorem~\ref{thm:transition} specializes to 
\begin{align}\label{eq:vexTransAcc}
\begin{split}
\mathfrak G_v&=\mathfrak G_{v_P}+(x_r-1)(\mathfrak G_{v_P}\star(1-t_{i,r})) \\
&=\mathfrak G_{v_P}+(x_r-1)(\mathfrak G_{v_P}-\mathfrak G_{v_C}) \\
&=x_r\mathfrak G_{v_P}+(1-x_r)\mathfrak G_{v_C}.
\end{split}
\end{align}
If no such index exists, then necessarily $(r,s)\in {\sf Dom}(v)$ and we have
\begin{align}\label{eq:vexTransNoAcc}
\begin{split}
\mathfrak G_v&=\mathfrak G_{v_P}+(x_r-1)(\mathfrak G_{v_P})\\
&=x_r\mathfrak G_{v_P}.
\end{split}
\end{align}

\begin{lemma}
\label{lemma:domdeg}
Fix any permutation $w$ and suppose the maximal corner $(r,s)\in {\sf Dom}(w)$. Then $\deg(\mathfrak G_w)=\deg(\mathfrak G_{w_P})+1.$
\end{lemma}	
\begin{proof}
This is an immediate consequence of Equation \eqref{eq:vexTransNoAcc} since multiplying any nonzero polynomial by $x_r$ increases the degree by 1.
\end{proof}

 Given a permutation $w$, the cell $(r,s)\in D(w)$ is called \mydef{accessible} if 
\begin{enumerate}
	\item $(r,s)\not\in {\sf Dom}(w)$ and
	\item there are no other cells which occur weakly southeast of $(r,s)$ in $D(w)$.
\end{enumerate}
The maximal corner is an accessible box if and only if there exists $i<r$ such that $\ell(w_P)+1=\ell(w_Pt_{i,r})$. 
For vexillary permutations, there can be at most one such $i$, so we define $v_C=v_Pt_{i,r}$ in this case.
We may construct $v_C$ graphically as follows. 
Consider the cells in $D(v)$ which sit weakly northwest of the accessible box in its connected component.  Move each of these diagonally one step in the northwest direction.  This new diagram is the $D(v_C)$.

\begin{lemma}
	\label{lemma:accdegree}
	Fix $v$ vexillary, where the maximal corner $(r,s)$ is an accessible box.  Then $\deg(\mathfrak G_v)=\max\{\deg(\mathfrak G_{v_P}),\deg(\mathfrak G_{v_C})\}+1$.
\end{lemma}
\begin{proof}
	The monomials of Grothendieck polynomials alternate in sign based on degree. As such, Equation \eqref{eq:vexTransAcc} is cancellation free.  Therefore, the top degree monomials in $\mathfrak G_v$ must come from $x_r\mathfrak G_{v_P}$ or $x_r\mathfrak G_{v_C}$.
\end{proof}

\begin{lemma}
\label{lemma:dommaximalcorner}
For $v$ vexillary, 
 if the maximal corner $(r,s)\in {\sf Dom}(v)$, then $\antidiagstat(v)=\antidiagstat(v_P)+1$.
\end{lemma}
\begin{proof}
By Equation~(\ref{eq:wPdef}), $\lambda(v_P)$ is obtained by removing the corresponding (boundary) cell  from $\lambda(v)$.  The label of this cell in $F_v$ is zero since $(r,s)\in {\sf Dom}(v)$.  At all other positions, $F_v$ matches $F_{v_P}$.  As such, $\tau_i(v)=\tau_i(v_P)$ for all $i>0$.  Therefore,
\begin{align*}
d(v)&= |\lambda(v)|+\sum_{i=1}^n \maxsizeantidiag(\tau_i(v))\\
&= |\lambda(v_P)|+1+\sum_{i=1}^n\maxsizeantidiag(\tau_i(v_P))\\
&=\antidiagstat(v_P)+1. \qedhere
\end{align*}
\end{proof}

\begin{lemma}
\label{lemma:accmaximal corner}
Fix $v$ vexillary and suppose the maximal corner $(r,s)$ is an accessible box.
\begin{enumerate}
\item  $\antidiagstat(v)\geq \antidiagstat(v_C)+1 $.
\item If $(r,s)$ is the only cell in its row within its connected component in $D(v)$, then $\antidiagstat(v)=\antidiagstat(v_C)+1$.
\item $\antidiagstat(v)\geq \antidiagstat(v_P)+1$.

\item If $(r,s)$ is not the only cell in its row within its connected component in $D(v)$, then $\antidiagstat(v)=\antidiagstat(v_P)+1$.
\end{enumerate}
\end{lemma}
\begin{proof}
Throughout, let $(a,b)$ denote the position of the box in $\lambda(v)$ which corresponds to $(r,s)$. Write $k=F_v(a,b)$.  By assumption since $(r,s)\not\in{\sf Dom}(v)$, $k\geq 1$.

\noindent (1) To get $F_{v_C}$ from $F_v$, take all labels weakly northwest of $(a,b)$ with label $k$ and decrease the value of these labels by $1$.  As such, $\tau_i(v)\supseteq \tau_i(v_C)$ for all $i$.  Furthermore, since $(r,s)$ has label $k$, $\tau_k(v)\supsetneq \tau_k(v_C)$.  
In particular,  $\tau_k(v_C)$ is obtained from $\tau_k(v)$ by removing a rectangular strip.  Since this strip contains $(r,s)$, removing this rectangle removes the last row of $\tau_k(v)$ entirely (and anything north of this row) by the definition of $(r,s)$.
Therefore, any antidiagonal path in $\tau_k(v_C)$ can be completed to a larger antidiagonal path in $\tau_k(v)$ by adding a box row $r$.  As such, $\maxsizeantidiag(\tau_k(v))>\maxsizeantidiag(\tau_k(v_C))$ and so
\begin{align*}
\antidiagstat(v)&=|\lambda(v)|+\sum_{i=1}^n \maxsizeantidiag(\tau_i(v))\\
&> |\lambda(v_C)|+\sum_{i=1}^n \maxsizeantidiag(\tau_i(v_C))\\
&=\antidiagstat(v_C).
\end{align*}
Since these are all integers, $\antidiagstat(v)\geq \antidiagstat(v_C)+1$.

\noindent (2) Since there is a single box in the same row as $(a,b)$ in $\tau_k(v)$ and this box is not in $\tau_k(v_C)$ (nor any boxes in its same column) we claim $\maxsizeantidiag(\tau_k(v))=\maxsizeantidiag(\tau_k(v_C))+1$.  For all other $i$, we have $\tau_i(v)=\tau_i(v_C)$ and so $\maxsizeantidiag(\tau_i(v))=\maxsizeantidiag(\tau_i(v_C))$.
  Therefore, $\antidiagstat(v)=\antidiagstat(v_C)+1$.

\noindent (3)  Using Equation~(\ref{eq:wPdef}), $F_{v_P}(i,j)=F_v(i,j)$ for all $(i,j)\in \lambda(v_P)$.  As such, 
\begin{equation}
\label{eqn:pivotequalities}
\tau_i(v_P)=
\begin{cases}
 \tau_i(v) &\text{if } i<k\\
\tau_i(v)-\{(a,b)\} &\text{otherwise}.
\end{cases}
\end{equation}
In particular, $\tau_i(v)\supseteq \tau_i(v_P)$ for all $i$.
 Therefore, \[\sum_{i=1}^n \maxsizeantidiag(\tau_i(v)) \geq \sum_{i=1}^n\maxsizeantidiag(\tau_i(v_P)).\]
Then
\begin{align*}
\antidiagstat(v)&=|\lambda(v)|+\sum_{i=1}^n \maxsizeantidiag(\tau_i(v))\\
&\geq |\lambda(v_P)|+1+\sum_{i=1}^n \maxsizeantidiag(\tau_i(v_P))\\
&=\antidiagstat(v_P)+1.
\end{align*}

\noindent (4) By assumption, $(r,s-1)\in D(v)$.  As such, if $(a,b)\in \tau_i(v)$ then $(a,b-1)\in \tau_i(v)$ as well.  Fix an antidiagonal path of cells in $\tau_i(v)$.  If it does not use $(a,b)$, then it is also an antidiagonal path of cells in $\tau_i(v_P)$.  If it does use $(a,b)$, then we can construct a new antidiagonal path of cells of the same size by replacing $(a,b)$ with $(a,b-1)$.  By (\ref{eqn:pivotequalities}), we see that this new antidiagonal path is also in $\tau_i(v_P)$.  As such, $\maxsizeantidiag(\tau_i(v))=\maxsizeantidiag(\tau_i(v))$.
Then we conclude $\antidiagstat(v)=\antidiagstat(v_P)+1$.
\end{proof}

\begin{proof}
Fix $v$ vexillary.  The statement is trivial for the identity, so assume $\ell(v)\geq 1$.  We will proceed by induction on the position of the maximal corner $(r,s)$ (ordering cells of the grid lexicographically).   In the base case, $v=21$, we confirm $\deg(\mathfrak G_v)=1=\antidiagstat(v)$.

  Assume the formula holds for any vexillary $v'$ whose maximal corner occurs before $(r,s)$, i.e., $\deg(\mathfrak G_{v'})=\antidiagstat(v')$.

\noindent Case 1: $(r,s)\in{\sf Dom}(v)$.
By Equation~(\ref{eq:wPdef}), the maximal corner of $v_P$ occurs before $(r,s)$.  Furthermore, $v_P$ is vexillary.  As such, 
\begin{align*}
\deg(\mathfrak G_v)&=\deg(\mathfrak G_{v_P})+1 &\text{(by Lemma~\ref{lemma:domdeg})}\\
&=\antidiagstat(v_P)+1 &\text{(by induction hypothesis)}\\
&=\antidiagstat(v) &\text{(by Lemma~\ref{lemma:dommaximalcorner})}.
\end{align*}

\noindent Case 2: $(r,s)\not\in{\sf Dom}(v)$ (i.e., it is an accessible box).

Both $v_P$ and $v_C$ are vexillary and their maximal corners (when defined) occur before $(r,s)$.
We know by Lemma~\ref{lemma:accdegree} and the induction hypothesis that 
\begin{equation}
\label{eqn:maxeq}
\deg(\mathfrak G_v)=\max\{\deg(\mathfrak G_{v_P}),\deg(\mathfrak G_{v_C})\}+1=\max\{\antidiagstat(v_P),\antidiagstat(v_C)\}+1.
\end{equation}
In particular, $1+\antidiagstat(v_P)\leq \deg(\mathfrak G_v)$ and $1+\antidiagstat(v_C)\leq \deg(\mathfrak G_v)$. Applying Lemma~\ref{lemma:accmaximal corner} to (\ref{eqn:maxeq}), we see that $\deg(\mathfrak G_v)\leq \antidiagstat(v)$.  
By parts (2) and (4) of Lemma~\ref{lemma:accmaximal corner}, since $(r,s)$ is an accessible box, $1+\antidiagstat(v_P)=\antidiagstat(v)$ or $1+\antidiagstat(v_C)=\antidiagstat(v)$. Then $\antidiagstat(v)=\deg(\mathfrak G_v)$.
\end{proof}

\section{Castelnuovo-Mumford regularity of Schubert determinantal ideals}
\label{section:CM}
We begin this section by recalling the connection between the Castelnuovo-Mumford regularity in the Cohen-Macaulay setting and the degree of a K-polynomial (Subsection~\ref{sect:CMreg}). We then  provide some background on Schubert determinantal ideals, explain how to  express Castelnuovo-Mumford regularity of Schubert determinantal ideals in terms of degrees of Grothendieck polynomials, and prove Theorems~\ref{thm:1stmainTheorem} and \ref{thm:2ndmainTheorem} (Subsection~\ref{sect:SchubBackground}).

\subsection{Castelnuovo-Mumford regularity and connections to K-polynomials}\label{sect:CMreg}

Let $S = \Bbbk[x_1,\ldots, x_n]$ be a polynomial ring over the field $\Bbbk$, and assume that $S$ is positively $\mathbb{Z}^d$-graded so that ${\rm deg}_{{\mathbf 0}}S = \Bbbk$. 
Let $M$ be a finitely generated graded $S$-module. The \mydef{multigraded Hilbert series} of $M$ is a formal power series in indeterminates $t_1,\dots, t_d$:
\[
H(M; \mathbf{t}) = \sum_{\mathbf{a}\in \mathbb{Z}^d}\text{dim}_{\Bbbk}(M_{\mathbf{a}})\mathbf{t}^{\mathbf{a}}
=\frac{K(M;\mathbf{t})}{\prod_{i=1}^n(1-\mathbf{t}^{\mathbf{a_i}})},~~~\text{deg}(x_i) = \mathbf{a_i}.
\]
The numerator $K(M;\bf{t})\in \Bbbk[\bf{t}^{\pm 1}]$ is called the $\mathbf{K}$\mydef{-polynomial} of $M$. 
When $S$ has the standard grading, that is $\text{deg}(x_i) = 1$, the K-polynomial is a Laurent polynomial in a single indeterminate $t$.

For the rest of this subsection, assume that $S$ has the standard grading, and let $I\subseteq S$ be a homogeneous ideal. 
There is a minimal free resolution 
\[
 0 \rightarrow \bigoplus_jS(-j)^{\beta_{l,j}(S/I)}\rightarrow \bigoplus_jS(-j)^{\beta_{l-1,j}(S/I)}\rightarrow \cdots \rightarrow \bigoplus_jS(-j)^{\beta_{0,j}(S/I)} \rightarrow S/I \rightarrow 0
\]
where $l\leq n$ and $S(-j)$ is the free $S$-module obtained by shifting the degrees of $S$ by $j$. 
The \mydef{Castelnuovo-Mumford regularity} of $S/I$, denoted $\reg(S/I)$, is defined as
\[
\reg(S/I):=\max\{j-i \,:  \, \beta_{i,j}(S/I)\neq 0\}.
\]
When $S/I$ is Cohen-Macaulay, we have that  
\begin{equation}\label{eq:mainRegEquation}
\reg(S/I) = \text{deg }K(S/I;t) - \text{ht}_S I,
\end{equation}
where $\text{ht}_S I$ denotes the height of the ideal $I$. 
See, for example, \cite[Lemma 2.5]{Benedetti.Varbaro} for justification of this formula. 
In this paper, we use Equation~\eqref{eq:mainRegEquation} to compute Castelnuovo-Mumford regularity of coordinate rings of certain matrix Schubert varieties and certain standard-graded Kazhdan-Lusztig varieties.

\subsection{Regularity of Schubert determinantal ideals and proofs of Theorems~\ref{thm:1stmainTheorem} and \ref{thm:2ndmainTheorem}}\label{sect:SchubBackground}

We begin by recalling basic facts about Schubert determinantal ideals. Fix an $n\times n$ permutation matrix $w$.  Let $X = (x_{ij})$ be an $n\times n$ matrix of distinct indeterminates, and let $X_{[p],[q]}$ denote the matrix formed by intersecting the first $p$ rows of $X$ and the first $q$ columns of $X$. Let $\Bbbk[{\bf x}] := \Bbbk[x_{ij}:1\leq i,j\leq n ]$. The \mydef{Schubert determinantal ideal} $I_w\subseteq \Bbbk[{\bf x}]$ is the ideal  
\[I_w = \langle \text{minors of size } r_w(i,j)+1 \text{ in } X_{[i],[j]}\,:  \, (i,j)\in \Ess(w) \rangle.\]
By \cite{Fulton.Flags}, $I_w$ is a prime ideal, and $\Bbbk[{\bf x}]/I_w$ is Cohen-Macaulay. Recall that $\Bbbk[{\bf x}]/I_w$ is the coordinate ring of the matrix Schubert variety $\overline{B_- w B_+}\subseteq \text{Mat}_{\Bbbk}(n,n)$ where $B_-\leq \GL_n(\Bbbk)$ is the Borel subgroup of invertible lower triangular matrices, $B_+\leq \GL_n(\Bbbk)$ is the Borel subgroup of invertible upper triangular matrices, and $\text{Mat}_{\Bbbk }(n,n)$ is the affine space of $n\times n$ matrices with entries in $\Bbbk$. Schubert determinantal ideals are homogeneous with respect to the standard grading of $\Bbbk[{\bf x}]$.

\begin{proof}[Proof of Theorems \ref{thm:1stmainTheorem} and \ref{thm:2ndmainTheorem}]
We first recall how to express the regularity of  $\Bbbk[{\bf x}]/I_w$ in terms of the degree of a Grothendieck polynomial. This was originally discussed in \cite{RRRSW}. By \cite{Fulton.Flags}, we have $\text{ht}_{\Bbbk[{\bf x}]} I_w=\#D(w)$. It then follows by \eqref{eq:mainRegEquation} that 
\[
\reg(\Bbbk[{\bf x}]/I_w) 
 = \text{deg } K(\Bbbk[{\bf x}]/I_w) - \#D(w).
\]
By \cite[Theorem~2.1]{Buc02} (see also \cite[Theorem~A]{KM}), $K(S/I_w;t)=\mathfrak G_w(1-t,\ldots,1-t)$. Furthermore,\[\text{deg }\mathfrak G_w(1-t,\ldots,1-t) = \text{deg }\mathfrak G_w(x_1,\ldots,x_n)\] 
since the the coefficients in the homogeneous components $\mathfrak G_w(x_1,\ldots,x_n)$ all have the same sign (see, for example, \cite{KM}). Thus, 
\begin{equation}\label{eq:schubReg}
\reg(\Bbbk[{\bf x}]/I_w)  = \text{deg } \mathfrak{G}_w(x_1,\dots, x_n) - \#D(w).
\end{equation}
Theorems~\ref{thm:1stmainTheorem} and \ref{thm:2ndmainTheorem} are now immediate from Theorems~\ref{thm:1432} and \ref{thm:vexDeg}.
\end{proof}

\section{Regularity of homogeneous Kazhdan-Lusztig ideals}\label{sec:vexGr}

In this section, we recall the basics of \emph{Kazhdan-Lusztig ideals} $J_{v,w}$ (Section \ref{sec:KLBackground}) and provide preliminary combinatorial formulas for regularity of Kazhdan-Lusztig ideals $J_{v,w}$ when $v$ is a $321$-avoiding permutation (Section \ref{sec:321RegPrelims}). We then provide an easily-computable  combinatorial formula for the regularity of open patches of Schubert varieties in Grassmannians (Section \ref{sec:regGrassPatches}). This proves a (generalization of a) conjecture from \cite{RRRSW} giving a correction to a conjecture of \cite{KLSS}.

\subsection{Kazhdan-Lusztig ideals}\label{sec:KLBackground}
We next recall Kazhdan-Lusztig ideals, which were introduced by A. Woo and A. Yong in \cite{WooYongSings} to study singularities of Schubert varieties. Given a permutation matrix $v\in S_n$, consider the matrix $M^{(v)}$ which has $1$'s at locations $(i,v_i)$,  indeterminate $z_{ij}$ in location $(i,j)\in D(v)$, and $0$'s elsewhere. Let $\Bbbk[{\bf z}^v]:= \Bbbk[z_{ij}\,:  \,(i,j)\in D(v)]$. Given $w\in S_n$, define the \textbf{Kazhdan-Lusztig ideal} $J_{v,w}\subseteq \Bbbk[{\bf z}^v]$ to be
\[
J_{v,w} = \langle \text{minors of size } r_w(i,j)+1 \text{ in } M^{(v)}_{[i],[j]}\,:  \, (i,j)\in \Ess(w)\rangle,
\]
which is not the unit ideal precisely when $w\leq v$ in Bruhat order. 
 The Kazhdan-Lusztig ideal $J_{v,w}$ is the prime defining ideal of the intersection of the Schubert variety $B_-\backslash \overline{B_- w B_+}\subseteq B_-\backslash \GL_n(\Bbbk)$ with the opposite Schubert cell $B_-\backslash B_-vB_-$ (see \cite[Corollary 3.3]{WooYongSings} and the preceding discussion). Furthermore, $\Bbbk[{\bf z}^{v}]/J_{v,w}$ is Cohen-Macaulay. This follows by \cite[Lemma A.4]{Kazhdan-Lusztig} together with the Cohen-Macaulayness of Schubert varieties \cite{Ramanathan}. See \cite[Section 3.2]{WooYongSings} for further discussion.

 Kahzdan-Lusztig ideals are not always homogeneous with respect to the standard grading on $\Bbbk[{\bf z}^{(v)}]$. However, when $v$ is 321-avoiding, and hence when $v$ is a Grassmannian permutation, $J_{v,w}$ is homogeneous with respect to the standard grading, see e.g., \cite[Footnote on pg. 25]{Knutson-Frob}.
Some further partial results on the problem of when Kazhdan-Lusztig ideals are homogeneous with respect to the standard grading \cite[Problem 5.5]{WooYongSings} can be found in the recent preprint \cite{Neye}.

\subsection{Preliminaries on regularity of Kazhdan-Lusztig ideals $J_{v,w}$ where $v$ is $321$-avoiding}\label{sec:321RegPrelims}
We next describe a formula for the regularity of $\Bbbk[{\bf z}^v]/J_{v,w}$ where $J_{v,w}$ is a standard-graded Kazhdan-Lusztig ideal. This formula will be in terms of $\mathfrak G_w({\mathbf x};{\mathbf y})$, a double Grothendieck polynomial. 
Let $G_w(\mathbf{x};\mathbf{y})$ denote the double Grothendieck polynomials in \cite{KM}, so that  $G_w(\mathbf{x};\mathbf{y})=\mathfrak G_w({\mathbf 1-\mathbf x};{\mathbf 1-\frac{\mathbf 1}{\mathbf y}})$. We also let $G_{v,w}({\bf t}) = \mathfrak{G}_{v,w}(1-{\bf t})$.

The torus $T^n$ acts on the opposite Schubert cell $B_-\backslash B_- v B_-$ by right multiplication. This induces a grading on $\Bbbk[{\bf z}^v]$ where variable $z_{ij}$ in the matrix $M^{(v)}$ has degree $e_{v(i)}-e_j$, where $e_i\in \mathbb{Z}^n$ denotes the $i^{\rm th}$ standard basis vector. By \cite[Theorem~4.5]{WooYongGrobner}, the $K$-polynomial of $\Bbbk[{\bf z}^v]/J_{v,w}$ for this $\mathbb{Z}^n$-grading is given by
\begin{equation}\label{eq:kPolyKLzn}
    K(\Bbbk[{\bf z}^v]/J_{v,w};{\bf t}) = G_w(t_{v(1)},\dots,t_{v(n)};t_1,\dots, t_n) = G_{v,w}(t_{ij}\mapsto t_{v(i)}/t_j).
\end{equation}

Note that the conventions in \cite{WooYongGrobner} differ from ours.

In the case where $v$ is 321-avoiding, there is a coarsening of the grading $f:\mathbb{Z}^n\rightarrow \mathbb{Z}$ which gives each $z_{ij}\in \Bbbk[z_{ij}]$ degree 1. Specifically, take $f(e_i) = 1$ if there exists $k>i$ such that $v^{-1}(k)<v^{-1}(i)$ and $f(e_i) = 0$ otherwise (see e.g., the footnote on page 25 of \cite{Knutson-Frob}). Then the $K$-polynomial of $\Bbbk[{\bf z}^v]/J_{v,w}$, with respect to the standard grading, is
\begin{align}\label{eq:kPolyKLst}
\begin{split}
  K(\Bbbk[{\bf z}^v]/J_{v,w};{ t}) &= G_w(t^{f(e_{v(1)})},\dots,t^{f(e_{v(n)})};t^{-f(e_1)},\dots, t^{-f(e_n)}) \\
    &= G_{v,w}(t_{ij}\mapsto t^{f(e_{v(i)})+f(e_j)}).  
\end{split}
\end{align}

\begin{example}
Let $v=34512$ and $w=21435$. Using Equation~\eqref{eq:unspecgroth}, we may compute $\mathfrak{G}_{v,w}({\bf t})=t_{11}t_{31}+t_{11}t_{22}-t_{11}t_{22}t_{31}$. For the $\mathbb{Z}^n$-grading, the substitution provided in Equation~\eqref{eq:kPolyKLzn} yields 
\[K(\Bbbk[{\bf z}^v]/J_{v,w};{\bf t}) =(1-\frac{t_{3}}{t_{1}})(1-\frac{t_{5}}{t_{1}})+(1-\frac{t_{3}}{t_{1}})(1-\frac{t_{4}}{t_{2}})-(1-\frac{t_{3}}{t_{1}})(1-\frac{t_{5}}{t_{1}})(1-\frac{t_{4}}{t_{2}}).\]
Using Theorem~\ref{thm:321svtGroth}, we may compute 
\[\mathfrak G_w({\mathbf x};{{\mathbf y}})=\frac{x_{1}}{y_{1}}\frac{x_{3}}{y_{1}}
+\frac{x_{1}}{y_{1}}\frac{x_{2}}{y_{2}}
+\frac{x_{1}}{y_{1}}\frac{x_{1}}{y_{3}}
-\frac{x_{1}}{y_{1}}\frac{x_{1}}{y_{3}}\frac{x_{2}}{y_{2}}
-\frac{x_{1}}{y_{1}}\frac{x_{1}}{y_{3}}\frac{x_{3}}{y_{1}}
-\frac{x_{1}}{y_{1}}\frac{x_{2}}{y_{2}}\frac{x_{3}}{y_{1}}
+\frac{x_{1}}{y_{1}}\frac{x_{1}}{y_{3}}\frac{x_{2}}{y_{2}}\frac{x_{3}}{y_{1}}.\]
Combining this with Equation~(\ref{eq:kPolyKLst}) yields
\[K(\Bbbk[{\bf z}^v]/J_{v,w};{\bf t}) =G_w(1,1,1,t,t;t^{-1},t^{-1},1,1,1)=2(1-t)^2-(1-t)^3\]
under the $\mathbb{Z}$-grading.
\end{example}

\begin{lemma}\label{lemma:321weight}
Let $v\in S_n$ such that $v$ is $321$-avoiding.  If $(i,j)\in D(v)$, then $f(e_j)=1$ and $f(e_{v_i})=0$.
\end{lemma}
\begin{proof}
Since $(i,j)\in D(v)$, there is $k=v_i$ such that $v_i>j$ and $i<v^{-1}_j$, thus $f(e_j)=1$.
If $f(e_{v_i})=1$, then there is $k>v_i$ such that $v^{-1}_k<i$. This would then imply that there is a $321$-pattern in $v$. In particular, we would have $v^{-1}_k<i<v^{-1}_j$, with $j<v_i<k$. As $v$ is $321$-avoiding, we conclude that $f(e_{v_i}) = 0$. 
\end{proof}

\begin{lemma}\label{lem:321KFormula}
Let $v,w\in S_n$ such that $v$ is $321$-avoiding and $w\leq v$.
Then 
\begin{equation}
   K(\Bbbk[{\bf z}^v]/J_{v,w};t) = G_{v,w}(t_{ij} \mapsto t) = \sum_{P\in \KPipes(v,w)}(-1)^{\#P-\ell(w)}(1-t)^{\#P}.
\end{equation}
\end{lemma}

\begin{proof}
 The coarsening of the grading $f:\mathbb{Z}^n\rightarrow \mathbb{Z}$ combined with
Lemma \ref{lemma:321weight} ensures that $t^{f(e_{v(i)})+f(e_j)} = t$ for $(i,j)\in D(v)$.  
Thus, the result follows by Equations \eqref{eq:unspecgroth} and \eqref{eq:kPolyKLst} together with the fact that $G_{v,w}({\bf t}) = \mathfrak{G}_{v,w}(1-{\bf t})$.
\end{proof}

We will use the following to prove the main result of this section (Theorem \ref{thm:KLSSCorrection}).  

\begin{proposition}
\label{prop:pipeKLreg}
Let $v,w\in S_n$ such that $v$ is $321$-avoiding and $w\leq v$.  Then,
\begin{equation}\label{eq:keyDegFormula}
   \deg K(\Bbbk[{\bf z}^v]/J_{v,w};t) = \deg \mathfrak{G}_{v,w}(\bf{t}).
\end{equation}
Furthermore, the Castelnuovo-Mumford regularity of $\Bbbk[{\bf z}^v]/J_{v,w}$ is given by
\begin{equation}\label{eq:regInTermsOfPipes}
    \reg (\Bbbk[{\bf z}^v]/J_{v,w}) = \deg \mathfrak{G}_{v,w}({\bf{t}})-\#D(w) =  \max\{\#P\mid P\in \KPipes(v,w)\}-\#D(w).
\end{equation}
\end{proposition}
\begin{proof}
Equation~\eqref{eq:keyDegFormula} is immediate from Lemma \ref{lem:321KFormula}. Equation~\eqref{eq:regInTermsOfPipes} follows from Equations \eqref{eq:keyDegFormula}, \eqref{eq:mainRegEquation} and the fact that $\text{ht}_{\Bbbk[{\bf z}^v]}J_{v,w} = \#D(w)$.
\end{proof}

\subsection{Castelnuovo-Mumford regularity of patches of Grassmannian Schubert varieties}\label{sec:regGrassPatches}
In \cite{RRRSW}, we gave a counterexample to a conjecture of Kummini-Lakshmibai-Sastry-Seshadri from \cite{KLSS} on the Castelnuovo-Mumford regularity of coordinate rings of standard open patches of certain Schubert varieties in Grassmannians. We then gave a conjecture of a correct formula \cite[Conjecture 5.6]{RRRSW}. In this short subsection, we prove a generalization of this conjecture. 

Identify the Grassmannian $\text{Gr}(k,n)$ with $P\backslash \GL_n(\Bbbk)$ where $P\subseteq \GL_n(\Bbbk)$ is the parabolic subgroup of block lower triangular matrices with block sizes $k$ and $n-k$ down the diagonal. Let $u$ and $g$ be a pair of Grassmannian permutations with descent at $k$. The Kazhdan-Lusztig ideal $J_{u,g}$ is the prime defining ideal of the intersection of the Schubert variety $P\backslash \overline{P g B_+}\subseteq P\backslash \GL_n(\Bbbk)$ with the open set $P\backslash P u B_-\subseteq P\backslash \GL_n(\Bbbk)$. The following theorem gives the regularity of the coordinate rings of these open sets of Grassmannian Schubert varieties.

\begin{theorem}\label{thm:KLSSCorrection}
Fix Grassmannian permutations $g$ and $u$ with descent at position $k$ so that $\lambda(g)\subseteq \lambda(u)$. Let $v$ be the vexillary permutation such that $D(v)=\maxexcited(\lambda(u),\lambda(g))$. 
 Then, \[\reg(\mathbb{\Bbbk}[{\bf z}^u]/J_{u,g})=\deg (\mathfrak{G}_v({\bf x})) - |\lambda(g)| = \sum_{i=1}^n \maxsizeantidiag(\tau_i(v)).\]
\end{theorem}
\begin{proof}
The first equality follows due to Equation~\eqref{eq:regInTermsOfPipes}, Theorem \ref{thm:grassToVex}, and the fact that $\text{ht}_{\Bbbk[{\bf z}^u]} J_{u,g} = |\lambda(g)|$. The second equality is then immediate by Theorem \ref{thm:vexDeg} and the fact that $|\lambda(g)| = \#D(v)$ by construction of $v$.
\end{proof}

We note that \cite[Conjecture 5.6]{RRRSW} concerned the special case of the above theorem where $u = (n-k+1)~(n-k+2)\dots n~1~2\dots (n-k)$, written in one line notation.

\section{Regularity of ladder determinantal ideals}\label{sec:regLad}

Our next goal is to provide a formula for the Castelnuovo-Mumford regularity of any \emph{one-sided ladder determinantal ideal}. Ladder determinantal ideals are generalized determinantal ideals which were introduced by S.~S. Abhyankar \cite{Abhyankar} to study singularities of Schubert varieties. There has since been substantial interest in their properties. For example, see \cite{Narasimhan, HT, Conca, Conca2, ConcaHerzog, GonLak, GM, KMY, Gorla, GMN, GK15} and references therein. The work of Ghorpade and Krattenthaler \cite{GK15} on $a$-invariants of certain ladder determinantal ideals is most closely related to our results. This is discussed in more detail at the end of Section \ref{sec:ladderBackground}.

\subsection{One-sided ladder determinantal ideals}\label{sec:ladderBackground}

A \emph{ladder} $L$ is a Young diagram (in English notation) filled with distinct indeterminates. Observe that a ladder is determined by a collection of southeast corners $L^{SE} = \{(a_i,b_i)\}_{i\in [s]}$ ordered northeast to southwest. Label the northwest corner of $L$ to be $(0,0)$. Take $(a_{s+1},b_{s+1})$ to be the southwestmost corner of the ladder and take $(a_1,b_1)$ be the northeastmost corner of the ladder.

Let $\mathcal{P}$ denote the lattice path from $(a_{s+1},b_{s+1})$ to $(a_1,b_1)$ which travels along the boundary of the ladder, so that cells weakly northwest of the $\mathcal{P}$ are in $L$ and boxes weakly southeast of $\mathcal{P}$ are not in $L$. 
Let $P = \{(c_j,d_j)\}_{j\in[s']}$ denote a collection of distinguished points along $\mathcal{P}$. To each $(c_j,d_j)\in P$, assign a value $r_j\in \mathbb{Z}_{> 0}$. Let $L_{I,J}$ denote the subset of $L$ with row indices in $I$ and column indices in $J$ for $I,J\subseteq[n]$. 

Let $\Bbbk[L]$ denote the polynomial ring generated by these indeterminate entries. 
Define the \emph{one-sided mixed ladder determinantal ideal} $I_{L,{\bf{r}}}$:
\[I_{L,{\bf{r}}} = \langle \text{minors of size } r_j \text{ in } L_{[c_j],[d_j]}\,:  \, j\in[s']\rangle\subseteq \Bbbk[L].\] 
Letting $I_j$ denote the ideal of $r_j\times r_j$ minors of $L_{[c_j],[d_j]}$, one observes that
\[
I_{L,{\bf r}} = \sum_{j\in[s']} I_j.
\]
Following \cite{KM}, we assume 
\begin{equation}\label{eq:laddRed}
    0<c_1-r_1<c_2-r_2<\dots<c_{s'}-r_{s'} \text{ and } 0<d_1-r_1<d_2-r_2<\dots<d_{s'}-r_{s'}
\end{equation}
 so that 
$I_j\subsetneq I_k$ for any $j\neq k$, $j,k\in [s']$. 
As outlined in \cite[Proposition 9.6]{Fulton.Flags}, $L$ can be identified with a vexillary matrix Schubert variety $\overline{X}_v$ where $\Ess{(v)}$ are the boxes indexed by $P$ and the ranks satisfy $r_v(c_j,d_j)=r_j-1$.

\begin{example}\label{ex:1sidelad}
To the left is a ladder $L$. Then $L^{\sf SE}=\{(5,3),(3,5)\}$ with marked points and corresponding ranks given in red. To the right is the associated permutation $v$.
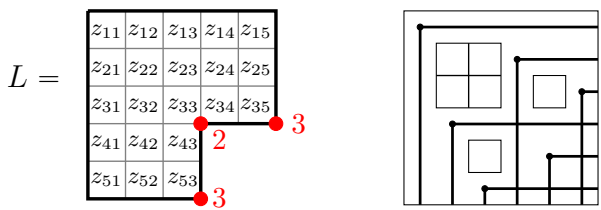
\[\begin{picture}(280,65)
\put(0,35){$L=$}
\put(30,-10){\begin{tikzpicture}[scale=.5]
\draw[line width = .1ex, gray] (0,4) -- (5,4);
\draw[line width = .1ex, gray] (0,3) -- (5,3);
\draw[line width = .1ex, gray] (0,2) -- (3,2);
\draw[line width = .1ex, gray] (0,1) -- (3,1);

\draw[line width = .1ex, gray] (1,5) -- (1,0);
\draw[line width = .1ex, gray] (2,5) -- (2,0);
\draw[line width = .1ex, gray] (3,5) -- (3,2);
\draw[line width = .1ex, gray] (4,5) -- (4,2);

\draw[line width = .25ex] (0,5)--(5,5)--(5,2)--(3,2)--(3,0)--(0,0)--(0,5);
\filldraw[red] (5,2) circle (1ex);
\filldraw[red] (3,2) circle (1ex);
\filldraw[red] (3,0) circle (1ex);
\put(1,62){\scriptsize{$z_{11}$}}
\put(15,62){\scriptsize{$z_{12}$}}
\put(30,62){\scriptsize{$z_{13}$}}
\put(44,62){\scriptsize{$z_{14}$}}
\put(57.5,62){\scriptsize{$z_{15}$}}
\put(1,48){\scriptsize{$z_{21}$}}
\put(15,48){\scriptsize{$z_{22}$}}
\put(30,48){\scriptsize{$z_{23}$}}
\put(44,48){\scriptsize{$z_{24}$}}
\put(57.5,48){\scriptsize{$z_{25}$}}
\put(1,34){\scriptsize{$z_{31}$}}
\put(15,34){\scriptsize{$z_{32}$}}
\put(30,34){\scriptsize{$z_{33}$}}
\put(44,34){\scriptsize{$z_{34}$}}
\put(57.5,34){\scriptsize{$z_{35}$}}
\put(1,20){\scriptsize{$z_{41}$}}
\put(15,20){\scriptsize{$z_{42}$}}
\put(30,20){\scriptsize{$z_{43}$}}
\put(1,6){\scriptsize{$z_{51}$}}
\put(15,6){\scriptsize{$z_{52}$}}
\put(30,6){\scriptsize{$z_{53}$}}
\put(77,25){\small{$\textcolor{red}3$}}
\put(47,19){\small{$\textcolor{red}2$}}
\put(47,-3){\small{$\textcolor{red}3$}}
\end{tikzpicture}}
\put(150,-10){\begin{tikzpicture}[scale=.43]
\draw (0,0) rectangle (6,6);

\draw (1,5) rectangle (3,3);
\draw[line width = .1ex] (1,4) -- (3,4);
\draw[line width = .1ex] (2,3) -- (2,5);

\draw (4,4) rectangle (5,3);
\draw (2,1) rectangle (3,2);

\filldraw (0.5,5.5) circle (.5ex);
\draw[line width = .2ex] (0.5,0) -- (0.5,5.5) -- (6,5.5);
\filldraw (3.5,4.5) circle (.5ex);
\draw[line width = .2ex] (3.5,0) -- (3.5,4.5) -- (6,4.5);
\filldraw (5.5,3.5) circle (.5ex);
\draw[line width = .2ex] (5.5,0) -- (5.5,3.5) -- (6,3.5);
\filldraw (1.5,2.5) circle (.5ex);
\draw[line width = .2ex] (1.5,0) -- (1.5,2.5) -- (6,2.5);
\filldraw (4.5,1.5) circle (.5ex);
\draw[line width = .2ex] (4.5,0) -- (4.5,1.5) -- (6,1.5);
\filldraw (2.5,0.5) circle (.5ex);
\draw[line width = .2ex] (2.5,0) -- (2.5,0.5) -- (6,0.5);
\end{tikzpicture}}
\end{picture}
\]
 Then 
 \begin{align*}
     I_{L,\bf r}&=\langle 3-\text{minors of } L_{[5],[3]}, 2-\text{minors of } L_{[3],[3]} ,3-\text{minors of } L_{[3],[5]}  \rangle\\
     &=\langle \det(L_{[3],\{3,4,5\}}), 2-\text{minors of } L_{[3],[3]} ,\det(L_{\{3,4,5\},[3]})  \rangle.  \qedhere
 \end{align*}
\end{example}

For certain one-sided mixed ladder determinantal ideals, regularity formulas can be deduced through $a$-invariant formulas of Ghorpade-Krattenthaler \cite{GK15}. Their formulas give results in the case in which $(r_1,r_2,\ldots,r_{s'})=(1,2,\ldots,t,t-1,\ldots,1)$ for some $t\in\mathbb{Z}_{>0}$, where Equation~(\ref{eq:laddRed}) is not imposed. Thus, for example, $L$ as in Example \ref{ex:1sidelad} is not in the class of ladders considered in \cite{GK15}. We note that an algorithm for $a$-invariant formulas is given in \cite{GK15} for two-sided mixed ladder determinantal ideals with the same restriction on ranks.

\subsection{One-sided ladder determinantal ideals via Grassmannian Kazhdan-Lusztig ideals}\label{sec:regLadder1}

We now recall that each one-sided ladder determinantal ideal is a Kazhdan-Lusztig ideal $\mathcal{N}_{u,g}$ where $u$ and $g$ are Grassmannian permutations. This was first shown by Gonciulea-Miller \cite[Theorem 4.7.3]{GM}; we include it here for completeness.

Take a ladder $L$ with $L^{\sf SE}=\{(a_i,b_i)\}_{i\in[s]}$, 
and marked points $P = \{(c_j,d_j)\}_{j\in[s']}$ assigning ranks $r_j$.
 Define $u\in S_{x+y}$ as the concatenation of partial permutations $u_i$, where for $i\in[s]$
\begin{align}
\label{eqn:oneSidedPatch}
\begin{split}
    u_{i}&= {\sf id}_{a_i-a_{i+1}}+b_{i}+a_{0}-a_{i}, \text{ and} \\
u_{s+1}&=[x+y]\setminus \cup_{i\in[s]}u_i.
\end{split}
\end{align}

Set 
$(c_0,d_0):=(a_0,b_0)$ with $r_0=1$ and 
$(c_{s'+1},d_{s'+1}):=(a_{s+1},b_{s+1})$ with $r_{s'+1}=1$.
Define $g\in S_{x+y}$ as the concatenation of partial permutations $g_i$, 
where for $i\in[s'+1]$,
\begin{align}
\label{eqn:oneSidedPatch2}
\begin{split}
    g_{i}&= {\sf id}_{k_i-k_{i-1}}+k_{i-1}+h_{i-1}, \text{ and} \\
g_{s'+2}&=[x+y]\setminus \cup_{i\in[s'+1]}g_i.
\end{split}
\end{align}
where $k_i=c_0-c_i+r_i-1$ and $h_i=d_i-r_i$.

Note that Equation~\eqref{eq:laddRed} and the assumption that each indeterminate appears in at least one minor ensure that $L(\code{(u)})=L(\code{(g)})$ and $u_j\geq g_j$ for each $j\in[x+y]$. 
Then 
by \cite[Theorem 4.7.3]{GM} we have the following:
\begin{proposition}\label{prop:KLladder}
Given a one-sided ladder determinantal ideal $I_{L,r}$ and $u,g$ as above,
$J_{u,g}$ and $I_{L,r}$ share the same generators.
\end{proposition}

\begin{example}\label{ex:ladVW}
For $L$ as in Example~\ref{ex:1sidelad}, below are $D(u)$ and $D(g)$ for the $u,g$ as defined in Equations (\ref{eqn:oneSidedPatch}) and (\ref{eqn:oneSidedPatch2}). 
\[
\begin{tikzpicture}[scale=.35]
\draw (0,0) rectangle (10,10);

\draw (0,10) rectangle (3,5);
\draw (7,8) rectangle (5,5);

\draw (0,9)--(3,9);
\draw (0,8)--(3,8);
\draw (0,7)--(3,7);
\draw (0,6)--(3,6);
\draw (1,10)--(1,5);
\draw (2,10)--(2,5);
\draw (3,10)--(3,5);

\draw (7,7)--(5,7);
\draw (7,6)--(5,6);
\draw (6,8)--(6,5);

\filldraw (3.5,9.5) circle (.5ex);
\draw[line width = .2ex] (3.5,0) -- (3.5,9.5) -- (10,9.5);
\filldraw (4.5,8.5) circle (.5ex);
\draw[line width = .2ex] (4.5,0) -- (4.5,8.5) -- (10,8.5);
\filldraw (7.5,7.5) circle (.5ex);
\draw[line width = .2ex] (7.5,0) -- (7.5,7.5) -- (10,7.5);
\filldraw (8.5,6.5) circle (.5ex);
\draw[line width = .2ex] (8.5,0) -- (8.5,6.5) -- (10,6.5);
\filldraw (9.5,5.5) circle (.5ex);
\draw[line width = .2ex] (9.5,0) -- (9.5,5.5) -- (10,5.5);
\filldraw (0.5,4.5) circle (.5ex);
\draw[line width = .2ex] (0.5,0) -- (0.5,4.5) -- (10,4.5);
\filldraw (1.5,3.5) circle (.5ex);
\draw[line width = .2ex] (1.5,0) -- (1.5,3.5) -- (10,3.5);
\filldraw (2.5,2.5) circle (.5ex);
\draw[line width = .2ex] (2.5,0) -- (2.5,2.5) -- (10,2.5);
\filldraw (5.5,1.5) circle (.5ex);
\draw[line width = .2ex] (5.5,0) -- (5.5,1.5) -- (10,1.5);
\filldraw (6.5,0.5) circle (.5ex);
\draw[line width = .2ex] (6.5,0) -- (6.5,0.5) -- (10,0.5);
\end{tikzpicture}
\hspace{3em} 
\begin{tikzpicture}[scale=.35]
\draw (0,0) rectangle (10,10);

\draw (2,5) rectangle (3,8);
\draw (2,7) -- (3,7);
\draw (2,6) -- (3,6);

\draw (5,5) rectangle (4,7);
\draw (4,6) -- (5,6);

\draw (6,5) rectangle (7,6);

\filldraw (0.5,9.5) circle (.5ex);
\draw[line width = .2ex] (0.5,0) -- (0.5,9.5) -- (10,9.5);
\filldraw (1.5,8.5) circle (.5ex);
\draw[line width = .2ex] (1.5,0) -- (1.5,8.5) -- (10,8.5);
\filldraw (3.5,7.5) circle (.5ex);
\draw[line width = .2ex] (3.5,0) -- (3.5,7.5) -- (10,7.5);
\filldraw (5.5,6.5) circle (.5ex);
\draw[line width = .2ex] (5.5,0) -- (5.5,6.5) -- (10,6.5);
\filldraw (7.5,5.5) circle (.5ex);
\draw[line width = .2ex] (7.5,0) -- (7.5,5.5) -- (10,5.5);
\filldraw (2.5,4.5) circle (.5ex);
\draw[line width = .2ex] (2.5,0) -- (2.5,4.5) -- (10,4.5);
\filldraw (4.5,3.5) circle (.5ex);
\draw[line width = .2ex] (4.5,0) -- (4.5,3.5) -- (10,3.5);
\filldraw (6.5,2.5) circle (.5ex);
\draw[line width = .2ex] (6.5,0) -- (6.5,2.5) -- (10,2.5);
\filldraw (8.5,1.5) circle (.5ex);
\draw[line width = .2ex] (8.5,0) -- (8.5,1.5) -- (10,1.5);
\filldraw (9.5,0.5) circle (.5ex);
\draw[line width = .2ex] (9.5,0) -- (9.5,0.5) -- (10,0.5);
\end{tikzpicture}
\]
\end{example}
As a consequence to Proposition~\ref{prop:KLladder}, the K-polynomial of each one-sided ladder determinantal ideal can be expressed both as a single Grothendieck polynomial and as a specialized double Grothendieck polynomial. Combining this with \cite{Fulton.Flags}, we have:

\begin{corollary} 
	Given a one-sided ladder $L$ with marked points $P = \{(c_j,d_j)\}_{j\in[s']}$ assigning ranks $r_j$, \[\reg(S/I_L)=\reg(S/J_{u,g})=\sum_{k=1}^n \maxsizeantidiag(\tau_k(v)),\]
	where $u,g$ are as defined in Equations (\ref{eqn:oneSidedPatch}) and (\ref{eqn:oneSidedPatch2}). 
	Here $v$ is the vexillary permutation such that $\Ess{(v)}$ are the boxes indexed by $P$ and $r_v(c_j,d_j)=r_j-1$.
\end{corollary}

\bibliographystyle{plainurl}
\bibliography{KLreg}

\begin{thebibliography}{10}

\bibitem{Abhyankar}
S.~S. Abhyankar.
\newblock {\em Enumerative combinatorics of {Y}oung tableaux}, volume 115 of
  {\em Monographs and Textbooks in Pure and Applied Mathematics}.
\newblock Marcel Dekker, Inc., New York, 1988.

\bibitem{Benedetti.Varbaro}
B.~Benedetti and M.~Varbaro.
\newblock On the dual graphs of {C}ohen-{M}acaulay algebras.
\newblock {\em Int. Math. Res. Not. IMRN}, 2015(17):8085--8115, 2015.

\bibitem{Buc02}
A.~S. Buch.
\newblock Grothendieck classes of quiver varieties.
\newblock {\em Duke Math. J.}, 115(1):75--103, 2002.

\bibitem{Buch}
Anders~Skovsted Buch.
\newblock A {L}ittlewood-{R}ichardson rule for the {$K$}-theory of
  {G}rassmannians.
\newblock {\em Acta Math.}, 189(1):37--78, 2002.

\bibitem{Conca}
A.~Conca.
\newblock Ladder determinantal rings.
\newblock {\em J. Pure Appl. Algebra}, 98(2):119--134, 1995.

\bibitem{Conca2}
A.~Conca.
\newblock Gorenstein ladder determinantal rings.
\newblock {\em J. London Math. Soc. (2)}, 54(3):453--474, 1996.

\bibitem{ConcaHerzog}
A.~Conca and J.~Herzog.
\newblock Ladder determinantal rings have rational singularities.
\newblock {\em Adv. Math.}, 132(1):120--147, 1997.

\bibitem{Fan.Guo}
Neil J.~Y. Fan and Peter~L. Guo.
\newblock Set-valued {R}othe tableaux and {G}rothendieck polynomials.
\newblock {\em Adv. in Appl. Math.}, 128:Paper No. 102203, 28, 2021.

\bibitem{FominKrillov}
S.~Fomin and A.~N. Kirillov.
\newblock Grothendieck polynomials and the {Y}ang-{B}axter equation.
\newblock In {\em Formal power series and algebraic combinatorics/{S}\'{e}ries
  formelles et combinatoire alg\'{e}brique}, pages 183--189. DIMACS,
  Piscataway, NJ, 1994.

\bibitem{Fulton.Flags}
W.~Fulton.
\newblock Flags, {S}chubert polynomials, degeneracy loci, and determinantal
  formulas.
\newblock {\em Duke Math. J.}, 65(3):381--420, 1992.

\bibitem{GK15}
S.~R. Ghorpade and C.~Krattenthaler.
\newblock Computation of the {$a$}-invariant of ladder determinantal rings.
\newblock {\em J. Algebra Appl.}, 14(9):1540014, 24, 2015.

\bibitem{GonLak}
N.~Gonciulea and V.~Lakshmibai.
\newblock Singular loci of ladder determinantal varieties and {S}chubert
  varieties.
\newblock {\em J. Algebra}, 229(2):463--497, 2000.

\bibitem{GM}
N.~Gonciulea and C.~Miller.
\newblock Mixed ladder determinantal varieties.
\newblock {\em J. Algebra}, 231(1):104--137, 2000.

\bibitem{Gorla}
E.~Gorla.
\newblock Mixed ladder determinantal varieties from two-sided ladders.
\newblock {\em J. Pure Appl. Algebra}, 211(2):433--444, 2007.

\bibitem{GMN}
E.~Gorla, J.~C. Migliore, and U.~Nagel.
\newblock Gr\"{o}bner bases via linkage.
\newblock {\em J. Algebra}, 384:110--134, 2013.

\bibitem{Graham.Kreiman}
W.~Graham and V.~Kreiman.
\newblock Excited {Y}oung diagrams, equivariant {$K$}-theory, and {S}chubert
  varieties.
\newblock {\em Trans. Amer. Math. Soc.}, 367(9):6597--6645, 2015.

\bibitem{Hafner}
E.~Hafner.
\newblock Vexillary {G}rothendieck polynomials via bumpless pipe dreams, 2022.
\newblock \href {http://arxiv.org/abs/2201.12432} {\path{arXiv:2201.12432}}.

\bibitem{HT}
J.~Herzog and N.~V. Trung.
\newblock Gr\"{o}bner bases and multiplicity of determinantal and {P}faffian
  ideals.
\newblock {\em Adv. Math.}, 96(1):1--37, 1992.

\bibitem{Kazhdan-Lusztig}
D.~Kazhdan and G.~Lusztig.
\newblock Representations of {C}oxeter groups and {H}ecke algebras.
\newblock {\em Invent. Math.}, 53(2):165--184, 1979.

\bibitem{Knutson-Frob}
A.~Knutson.
\newblock Frobenius splitting, point-counting, and degeneration, 2009.
\newblock \href {http://arxiv.org/abs/0911.4941} {\path{arXiv:0911.4941}}.

\bibitem{KM}
A.~Knutson and E.~Miller.
\newblock Gr\"obner geometry of {S}chubert polynomials.
\newblock {\em Ann. of Math.}, 161(3):1245--1318, 2005.

\bibitem{KMY}
A.~Knutson, E.~Miller, and A.~Yong.
\newblock Gr\"obner geometry of vertex decompositions and of flagged tableaux.
\newblock {\em J. Reine Angew. Math.}, 630:1--31, 2009.

\bibitem{KLSS}
M.~Kummini, V.~Lakshmibai, P.~Sastry, and C.~S. Seshadri.
\newblock Free resolutions of some {S}chubert singularities.
\newblock {\em Pacific J. Math.}, 279(1-2):299--328, 2015.

\bibitem{Lascoux.Transition}
A.~Lascoux.
\newblock Transition on {G}rothendieck polynomials.
\newblock In {\em Physics and combinatorics, 2000 ({N}agoya)}, pages 164--179.
  World Sci. Publ., River Edge, NJ, 2001.

\bibitem{LS82}
A.~Lascoux and M.~P. Sch\"{u}tzenberger.
\newblock Structure de {H}opf de l'anneau de cohomologie et de l'anneau de
  {G}rothendieck d’une vari\'et\'e de drapeaux.
\newblock {\em C.~R.~Acad.~Sci. Paris S\'{e}r.~I Math.}, 295(11):629--633,
  1982.

\bibitem{Le00}
C.~Lenart.
\newblock Combinatorial aspects of the {K}-theory of {G}rassmanians.
\newblock {\em Ann. Comb.}, 4(1):67--82, 2000.

\bibitem{Manivel}
L.~Manivel.
\newblock {\em Symmetric functions, {S}chubert polynomials and degeneracy
  loci}.
\newblock SMF/AMS Texts and Monographs. American Mathematical Society,
  Providence, 2001.
\newblock Translated from the 1998 French original by John R. Swallow.

\bibitem{MPP}
A.~H. Morales, I.~Pak, and G.~Panova.
\newblock Hook formulas for skew shapes {I}. {$q$}-analogues and bijections.
\newblock {\em J. Combin. Theory Ser. A}, 154:350--405, 2018.

\bibitem{Narasimhan}
H.~Narasimhan.
\newblock The irreducibility of ladder determinantal varieties.
\newblock {\em J. Algebra}, 102(1):162--185, 1986.

\bibitem{Neye}
E.~Neye.
\newblock A {G}r\"obner basis for {S}chubert patch ideals, 2021.
\newblock \href {http://arxiv.org/abs/2111.13778} {\path{arXiv:2111.13778}}.

\bibitem{PSW}
O.~Pechenik, D.~Speyer, and A.~Weigandt.
\newblock {Castelnuovo-Mumford regularity of matrix Schubert varieties}, 2021.
\newblock \href {http://arxiv.org/abs/2111.10681} {\path{arXiv:2111.10681}}.

\bibitem{RRRSW}
J.~Rajchgot, Y.~Ren, C.~Robichaux, A.~St.~Dizier, and A.~Weigandt.
\newblock Degrees of symmetric {G}rothendieck polynomials and
  {C}astelnuovo-{M}umford regularity.
\newblock {\em Proc. Amer. Math. Soc.}, 149(4):1405--1416, 2021.

\bibitem{Ramanathan}
A.~Ramanathan.
\newblock Schubert varieties are arithmetically {C}ohen-{M}acaulay.
\newblock {\em Invent. Math.}, 80(2):283--294, 1985.

\bibitem{Weigandt.BPD}
A.~Weigandt.
\newblock Bumpless pipe dreams and alternating sign matrices.
\newblock {\em J. Combin. Theory Ser. A}, 182:Paper No. 105470, 52, 2021.

\bibitem{WooYongSings}
A.~Woo and A.~Yong.
\newblock Governing singularities of {S}chubert varieties.
\newblock {\em J. Algebra}, 320(2):495--520, 2008.

\bibitem{WooYongGrobner}
A.~Woo and A.~Yong.
\newblock A {G}r\"obner basis for {K}azhdan-{L}usztig ideals.
\newblock {\em Amer. J. Math.}, 134(4):1089--1137, 2012.

\bibitem{YongReg}
A.~Yong.
\newblock Castelnuovo-{M}umford regularity and {S}chubert geometry, 2022.
\newblock \href {http://arxiv.org/abs/2202.06362} {\path{arXiv:2202.06362}}.

\end{thebibliography}

\end{document}